\documentclass[12pt]{amsart}

\usepackage{amsmath,amssymb,amsthm,verbatim,color}
\usepackage{graphicx}
\usepackage{hyperref}

\newtheorem{theorem}[subsection]{Theorem}
\newtheorem{lemma}[subsection]{Lemma}
\newtheorem{corollary}[subsection]{Corollary}
\newtheorem{prop}[subsection]{Proposition}

\newtheorem{question}[subsection]{Question}

\theoremstyle{definition}
\newtheorem{definition}[subsubsection]{Definition}
\newtheorem{remark}[subsection]{Remark}
\newtheorem{example}[subsection]{Example}

\newcommand{\dist}{\mathrm{dist}}
\newcommand{\vecspan}{\mathrm{span}}
\newcommand{\Sect}{\mathrm{Sect}}
\newcommand{\haus}{\mathcal{H}}
\newcommand{\leb}{\mathcal{L}}
\newcommand{\spt}{\mathrm{spt}}
\newcommand{\diam}{\mathrm{diam}}
\newcommand{\Ric}{\mathrm{Ric}}
\newcommand{\reg}{\mathrm{reg}}
\newcommand{\sign}{\mathrm{sign}}
\newcommand{\spine}{\mathrm{spine}}
\newcommand{\lb}{\left}
\newcommand{\rb}{\right}
\newcommand{\graph}{\mathrm{graph}}
\newcommand{\tr}{\mathrm{tr}}
\newcommand{\Scal}{\mathrm{Scal}}
\newcommand{\eps}{\epsilon}
\newcommand{\proj}{\mathrm{proj}}
\newcommand{\sff}{\mathrm{II}}
\newcommand{\sn}{\mathrm{sn}}
\newcommand{\sing}{\mathrm{sing}}
\newcommand{\cF}{\mathcal{F}}
\newcommand{\cG}{\mathcal{G}}
\newcommand{\R}{\mathbb{R}}
\newcommand{\N}{\mathbb{N}}
\newcommand{\bC}{\mathbf{C}}
\newcommand{\C}{\mathbb{C}}
\newcommand{\bM}{\mathbf{M}}
\newcommand{\del}{\partial}
\newcommand{\delbar}{\bar\partial}
\newcommand{\cB}{\mathcal{B}}
\newcommand{\mc}{\mathbf{H}}
\newcommand{\Z}{\mathbb{Z}}
\newcommand{\cE}{\mathcal{E}}

\newcommand{\Ran}{\mathrm{Ran}}
\newcommand{\cK}{\mathcal{K}}
\newcommand{\cM}{\mathcal{M}}
\newcommand{\mcfK}{\mathcal{K}}
\newcommand{\mcfM}{\mathcal{M}}
\newcommand{\dilD}{\mathcal{D}}
\newcommand{\mcfC}{\mathcal{C}}
\newcommand{\vel}{\mathbf{vel}}
\newcommand{\bH}{\mathbf{H}}
\newcommand{\cU}{\mathcal{U}}
\newcommand{\cI}{\mathcal{I}}
\newcommand{\cL}{\mathcal{L}}

\title{Regularity of minimal surfaces near quadratic cones}
\author{Nick Edelen and Luca Spolaor}

\address{Department of Mathematics, University of Notre Dame, Notre Dame, IN, 46556}
\email{nedelen@nd.edu}
\address{Department of Mathematics, University of California, San Diego, La Jolla, CA, 92093}
\email{lspolaor@ucsd.edu}

\thanks{N.E. was supported in part by NSF grant DMS-1606492. L.S. was partially supported by NSF grant DMS-1810645.}

\begin{document}

\begin{abstract}
Hardt-Simon \cite{HaSi} proved that every area-minimizing hypercone $\bC$ having only an isolated singularity fits into a foliation of $\R^{n+1}$ by smooth, area-minimizing hypersurfaces asymptotic to $\bC$.  In this paper we prove that if a stationary varifold $M$ in the unit ball $B_1 \subset \R^{n+1}$ lies sufficiently close to a minimizing quadratic cone (for example, the Simons' cone $\bC^{3,3}$), then $\spt M \cap B_{1/2}$ is a $C^{1,\alpha}$ perturbation of either the cone itself, or some leaf of its associated foliation.  In particular, we show that singularities modeled on these cones determine the local structure not only of $M$, but of any nearby minimal surface.  Our result also implies the Bernstein-type result \cite{SiSo}, which characterizes area-minimizing hypersurfaces asymptotic to a quadratic cone as either the cone itself, or some leaf of the foliation.

\end{abstract}

\maketitle
\tableofcontents

\section{Introduction}

In this paper we are interested in the following question:
\begin{question}\label{q:main}
Suppose $M_i$ is a sequence of minimal $n$-dimensional surfaces, converging to some $\tilde M$ with multiplicity one in the unit ball in $\R^{n+1}$.  How does the singular structure of $\tilde M$ determine the singular or regular structure of the $M_i$ in $B_{1/2}$?
\end{question}

This question and its variants underly significant amount of research in the field of minimal surfaces, and other variational problems. Question \ref{q:main} arises when attempting to study the singularity structure or compactness properties of some class of surfaces.  For example, often the $M_i$ form some kind of ``blow-up'' sequence for a singularity, and the resulting $\tilde M$ is a singularity model.  In the very important special case when $M_i$ are dilations around a fixed point of a given minimal $M$, i.e. when
\[
M_i = \lambda_i (M - x), \quad \lambda_i \to \infty,
\]
then $\tilde M$ is dilation invariant, and any such $\tilde M$ arising in this fashion is called a tangent cone of $M$ at $x$.

Question \ref{q:main} is entirely answered when $\tilde M$ is smooth: in this case Allard's theorem \cite{All} implies that for $i$ large, the $M_i \cap B_{1/2}$ must be smooth also (in fact must be $C^{1,\alpha}$ perturbations of $\tilde M$).

For singular $\tilde M$, Question \ref{q:main} has been answered under some structral assumptions.  When $\tilde M$ has at most a ``strongly isolated'' singularity\footnote{Strongly isolated here means that some tangent cone of $M$ at the singularity is a multiplicity-one cone with an isolated singularity at $0$.}, \emph{and} each $M_i$ has at least one singularity of the same type, then profound work of \cite{Simon0} shows that the $M_i \cap B_{1/2}$ must be $C^1$ perturbations of $\tilde M$ for large $i$.  The conditions on $M_i$ are naturally satisfied when $\tilde M$ is the tangent cone of some minimal surface, and $M_i$ are the dilations around a fixed point.

In certain particular cases one can use the topology of $\tilde M$ to deduce singular structure on $M_i$.  For example, when $\tilde M$ is a union of half-planes, then \cite{Simon1} showed that $M_i \cap B_{1/2}$ are a $C^{1,\alpha}$ perturbation of the $\tilde M$.  Similar results hold if $\tilde M$ has tetrahedral singularities, and the $M_i$ have an associated ``orientation structure'' (\cite{CoEdSp}); or when $\tilde M$ is a union of two planes, and the $M_i$ are $2$-valued graphs (\cite{becker-kahn2017}).

Notice that in these theorems, either by assumption or by the nature of $\tilde M$, all the $M_i$ have the same regular or singular structure as $\tilde M$.

\vspace{3mm}

Question \ref{q:main} becomes more subtle when one does not assume anything about the singular nature of the $M_i$.  In this case it is possible for a sequence of smooth minimal (even area-minimizing!) surfaces to limit to a singular one.  For example, \cite{BDG} have constructed a foliation by smooth minimal surfaces of the complement of the Simons's cone
\[
\bC^{3,3} = \{ (x, y) \in \R^8 : |x| = |y| \} .
\]
More generally, \cite{HaSi} showed the same holds for any area-minimizing cone $\bC^n \subset \R^{n+1}$, having an isolated singularity.  In these circumstances the $M_i$ need not have the same singularity structure as $\tilde M$.

More generally, it is conceivable that singularities of one type can limit to a singularity of a different type.  Or, even worse, that multiple singularities of various types or dimensions could collapse into a singularity of some other type.  One can build toy examples of these behaviors using geodesic nets, but to our knowledge no examples exist for surfaces of dimension $> 1$.

\vspace{5mm}

In this paper we answer Question \ref{q:main} in the case when $\tilde M$ has singularities modeled on area-minimizing quadratic cones, i.e. so called minimizing $(p, q)$-singularities (see Definition \ref{def:pq}).  One of the main results of this paper is that any singularity in \emph{any minimal surface} that is sufficiently nearby a minimizing $(p, q)$-singularity, must be of the same type, and in particular none of the pathologies of the previous paragraph can occur for these singularities.

\begin{definition}\label{def:pq}
Following \cite{SiSo}, define the $\bC^{p,q}$ quadratic minimal cone as the hypersurface
\[
\bC^{p,q} = \{ (x, y) \in \R^{p+1} \times \R^{q+1} : q|x|^2 = p |y|^2 \} \subset \R^{n+1},
\]
where $p, q > 0$ are integers, and $p + q = n-1$.  Some people refer to the cones $\bC^{p,q}$ as generalized Simons' cones.

Given a surface or varifold $M$, we say that $x \in\spt M$ is a $(p,q)$-singularity if some tangent cone of $M$ at $x$ is (up to rotation) equal to $\bC^{p,q}$ with multiplicity one.  We say $x$ is a minimizing $(p,q)$-singularity if the associated cone $\bC^{p,q}$ is area-minimizing.
\end{definition}

Some remarks are in order.
\begin{remark}
An easy computation shows each $\bC^{p,q}$ is minimal.  However, there are no non-flat area-minimizing hypercones in $\R^{n+1}$, for $n < 7$, and of course by dimension-reducing there are no singular area-minimizing hypersurfaces in these dimensions either.  When $n = 7$, the cones $\bC^{3,3}$ and $\bC^{2,4}$ are area-minimizing, and in fact up to rigid motions these are the only known area-minimizing hypercones in $\R^8$.  When $n > 7$, then every $\bC^{p,q}$ cone is area-minimizing.  See, e.g. \cite{Simoes} and the references therein.  Thus our results are motivated by, and most relavant to, the regularity theory for area-minimizing hypersurfaces.
\end{remark}

\begin{remark}
By work of \cite{AllAlm}, if $x$ is a $(p,q)$-singularity of an $M$ which is stationary (or has $L^p$ mean curvature, for $p > n$), then $M$ is nearby a $C^{1,\alpha}$ perturbation of $\bC^{p,q}$.
\end{remark}

As a corollary to our main theorem, we obtain the following answer to \ref{q:main}.
\begin{theorem}\label{thm:teaser}
Let $M$ be a multiplicity-one, stationary integral $n$-varifold in $B_1$, which is regular away from $0$, and has a minimizing $(p, q)$-singularity at $0$.  Let $M_i$ be a sequence of stationary, integral varifold in $B_1$, so that $M_i \to M$ as varifolds.  Then for each $i$ sufficiently large, $M_i \llcorner B_{1/2}$ has either an isolated singularity of the same type $(p,q)$, or is entirely regular.

More precisely, we have the following.  Let $S_+$, $S_-$ be leaves of the foliation by minimal surfaces of $\R^{n+1} \setminus \bC^{p,q}$ (as constructed by \cite{HaSi}) which lie in different connected components.  Let $M_{\pm}$ be the smooth manifolds obtained by replacing a small neighborhood of $0$ in $M_0 := \spt M$ with a scaled-down copy of $S_\pm$.  Then, for $i$ sufficiently large, $\spt M_i \cap B_{1/2}$ is a $C^{1,\alpha}$ perturbation of one of $M_0$, $M_+$, $M_-$.
\end{theorem}

We remark that, only from the information that $M_i \to M$, one cannot distinguish a priori whether each $M_i$ is regular or singular.  See Section \ref{sec:main} for a more detailed statement of our main regularity theorem, and for other corollaries.

The main novelty of Theorem \ref{thm:main} is that, unlike previous regularity results for minimal surfaces near isolated singularities (e.g. \cite{AllAlm}, \cite{Simon0}), we do not prescribe a priori the density of the $M_i$ at any point, that is we do not impose them to be singular at the origin, nor at any other point.  As a consequence, even if a minimal surface is close at scale $1$ to a cone with an isolated singularity, the surface itself may be entirely smooth.

We can give a further characterization of the $M_i$ when they are singular: in this case the $M_i$ must be one of the examples of minimal surfaces as constructed by \cite{CaHaSi}.  Finally, we can use our regularity theorem to reprove the rigidity result of \cite{SiSo}, which characterize complete minimal surfaces asymptotic to quadratic cones.

It would be interesting to know whether our results carry over to other (area-minimizing) singularity models.  Unfortunately, the only other known area-minimizing hypercones are of so-called isoparametric type, and for these other examples our techniques do not seem to work. See Section 4 of \cite{SiSo} for further discussions.


\section{Notation and preliminaries}

We work in $\R^{n+1}$.  We denote by $\haus^n$ the $n$-dimensional Hausdorff measure.  Given a subset $A \subset \R^{n+1}$, we let $d_A(x) = \inf_{a \in A} |x - a|$ be the Euclidean distance to $A$, and given $r > 0$ write
\[
B_r(A) := \{ x \in \R^{n+1} : d_A(x) < r \}
\]
for the (open) $r$-tubular neighborhood of $A$.  Similarly $B_r(x)$ is the open $r$-ball centered at $x \in \R$.  If $x = 0$, we may sometimes just write $B_r$.  We write $\overline{B_r(A)}$ for the closed tubular neighborhood or $r$-ball.  We set $\omega_n = \haus^n(\R^n \cap B_1)$ to be the volume of the $n$-dimensional unit ball.  We define the translation/dilation map $\eta_{x, r}(z) := (z - x)/r$.

We may occasionally use the notation of Cheeger: we denote with $\Psi(\eps_1, \ldots, \eps_k | c_n, \ldots, c_N)$ a non-negative function, which for any fixed $c_1, \ldots, c_N$, satisfies
$$
\lim_{\eps_1, \ldots, \eps_k \to 0} \Psi = 0.
$$

We shall always treat graphing functions as scalars.  Given oriented hypersurfaces $M, N$, and an open subset $U \subset N$, if we write
\[
M = \graph_{N}(u)
\]
we mean that $M = \{ x + u(x) \nu_N(x) : x \in U \subset N \}$, where $\nu_N$ is the unit normal of $N$.

Given an (oriented) hypersurface $N$, a function $u : N \to \R^n$, and a $\beta \in (0, 1)$, we define the Holder semi-norm
\[
[u]_{\beta, r} = \sup_{x \neq y \in N \cap B_r \setminus B_{r/2}} \frac{|u(x) - u(y)|}{|x - y|},
\]
and, given an integer $k \geq 0$, the Holder norm
\[
|u|_{k, \beta, r} = [D^k u]_{\beta, r} + \sum_{i=0}^k r^{1-k} \sup_{x \in N \cap B_r \setminus B_{r/2}} |\nabla^i u|.
\]
Typically when we use these norms, $N$ will be conical or nearly conical.

\subsection{Varifolds and first variation}

We are concerned with integral $n$-varifolds that are stationary, or have $L^p$ mean curvature.  Recall that an $n$-varifold $M$ is integral if it has the following structure: there is a countably $n$-rectifiable set $\tilde M$, and a $\haus^n \llcorner \tilde M$-integrable, non-negative $\Z$-valued function $\theta$, so that
\[
M(\phi(x, V)) = \int_{\tilde M} \phi(x, T_x \tilde M) \theta(x)  \,d\haus^n(x), \quad \forall \phi \in C^0_c(\R^{n+1} \times Gr(n, n+1)).
\]
Here $Gr(n, n+1)$ is the Grassmanian bundle, i.e. the space of unoriented $n$-planes in $\R^{n+1}$.  If $N$ is an $n$-manifold, then $N$ induces a natural $n$-varifold in the obvious fashion, which we write as as $[N]$.  Given a proper, $C^1$ map $\eta$, we write $\eta_\sharp M$ for the pushforward of $M$.

We write $\mu_M$ for the mass measure of $M$.  The first variation of $M$ in an open subset $U \subset \R^{n+1}$ is the linear functional
\[
\delta M(X) = \int \mathrm{div}_M(X) d\mu_M, \quad X \in C^1_c(U, \R^{n+1}),
\]
where $\mathrm{div}_M(X)$ is defined for $\mu_M$-a.e. $x$ as follows: if $e_i$ is an orthonormal basis for the tangent space $T_xM$, then
\[
\mathrm{div}_M(X) = \sum_i \langle e_i, D_{e_i} X \rangle.
\]

A integral $n$-varifold $M$ (or surface) is stationary in $U \subset \R^{n+1}$ if $\delta M(X) = 0$ for all $X$ compactly supported in $U$.  $M$ is said to have generalized mean curvature $H_M$, and zero generalized boundary in $U$, if 
\[
\delta M(X) = -\int H \cdot X d\mu_M \quad \forall X \in C^1_c(U, \R^{n+1}),
\]
where $H$ is some $\mu_M$-integrable vector field.

Let $M$ have generalized mean curvature $H_M$ in $B_1$, and zero generalized boundary, and suppose $||H_M||_{L^p(B_1; \mu_M)} \leq \Lambda < \infty$ for some $p > n$.  Then $M$ admits the area monotonicity (see \cite{All})
\begin{equation}\label{eqn:monotonicity}
\left( \frac{\mu_M(B_s(x))}{s^n} \right)^{1/p} \leq \frac{\Lambda}{p-n} (r^{1-n/p} - s^{1-n/p}) + \left( \frac{\mu_M(B_r(x))}{r^n} \right)^{1/p} .
\end{equation}
for any $x \in B_1$, and $0 < s < r < 1-|x|$.  Of course if $M$ is stationary, then $r^{-n} \mu_M(B_r(x))$ is increasing for all $r < 1-|x|$.

Further, Allard's theorem \cite{All} implies $M$ as in the previous paragraph admits the following regularity: there is a $\delta(n, p)$ so that if for some $n$-plane $V^n$ we have
\begin{gather}
\int_{B_1} d_V^2 \,d\mu_M + ||H_M||_{L^p(B_1; \mu_M)}^2 \leq E \leq \delta^2, \notag\\
\mu_M(B_1) \leq 3/2 \omega_n \quad\mbox{and}\quad \mu_M(B_{1/10}) \geq (1/2) \omega_n (1/10)^n,\notag
\end{gather}
then there is a $C^{1,1-n/p}$ function $u : V \cap B_{1/2} \to V^\perp$, so that 
\[
\spt M\cap B_{1/2} = \graph(u) \cap B_{1/2}, \quad |u|_{C^{1,1-n/p}} \leq c(n) E^{1/2}.
\]

\subsection{Jacobi fields}

Given a smooth, oriented minimal hypersurface $N^n$, let us write $\cM_N$ for the mean curvature operator on graphs over $N$, i.e. so that given $u : U \subset N \to \R$, and $x \in U$, then $\cM_N(u)(x)$ denotes the mean curvature of $\graph_U(u)$ at the point $x + u(x)\nu_N(x)$.  Equivalently, $-\cM_N$ is the Euler-Lagrange operator for the area functional on graphs over $N$.  $\cM_N$ is a second-order, quasi-linear elliptic operator
\[
\cM_N(u) = a_N(x, u, \nabla u)^{ij} \nabla^2_{ij} u + b_N(x,u, \nabla u),
\]
whose coefficients $a_N(x, z, p)$, $b_N(x, z, p)$ depend smoothly on $x, z, p$ and the submanifold $N$.

Write $\cL_N$ for the linearization of $\cM_N$ at $u = 0$.  $\cL_N$ is called the Jacobi operator, and any solution $w$ to $\cL_N(w) = 0$ is called a Jacobi field.  $\cL_N$ is a linear, elliptic operator:
\[
\cL_N = \Delta_N + |A_N|^2.
\]
Here $A_N$ is the second fundamental form of $N \subset \R^{n+1}$, and $\Delta_N$ is the connection Laplacian.  If $\phi_t$ is a family of compactly supported diffeomorphisms of $\R^{n+1}$, and $\del_t \phi_t|_{t = 0, x \in N} = f \nu_N$ on $N$, then 
\[
\frac{d^2}{dt^2}\Big|_{t=0} {\rm vol}(\phi_t(N)) = - \int_N f \cL_N f \, d\haus^n.
\]
$N$ is called stable if $\cL_N \leq 0$ when restricted to any compact subset of $N$.

\vspace{5mm}

When $N = \bC$ is a cone, with smooth, compact cross section $\Sigma = \bC \cap S^n$, then we can further decompose
\[
\cL_N = \del_r^2 + (n-1)r^{-1} \del_r + r^{-2} \cL_\Sigma, \quad \cL_\Sigma = \Delta_\Sigma + |A_{\Sigma}|^2,
\]
where $r = |x|$ is the radial distance, $\omega = x/|x|$, and $\Delta_\Sigma$, $A_\Sigma$ are the connection Laplacian, second fundamental form (resp.) of $\Sigma \subset S^n$.

Since $\Sigma$ is compact, there is $L^2(\Sigma)$-orthonormal basis of eigenfunctions $\phi_i$ of $\cL_\Sigma$, with corresponding eigenvalues $\mu_1 < \mu_2 \leq \cdots \to \infty$:
\[
\cL_\Sigma \phi_i + \mu_i \phi_i = 0, \quad \int_\Sigma \phi_i \phi_j d\haus^{n-1} = \delta_{ij}.
\]

By the Rayleigh quotient $\mu_1 \leq -(n-1)$.  On the other hand, when $\bC$ is stable we have $\mu_1 \geq -((n-2)/2)^2$, and we have strict inequality when $\bC$ is strictly stable (see \cite{CaHaSi}).  If we define
\[
\gamma^{\pm}_i = -((n-2)/2) \pm \sqrt{ ((n-2)/2)^2 + \mu_i },
\]
then for any solution $w$ to $\cL_\bC(w) = 0$, with $\bC$ being strictly stable, we can expand in $L^2_{loc}(\bC)$
\[
w(r\omega) = \sum_{i=1}^\infty (a_i^+ r^{\gamma_i^+} + a_i^- r^{\gamma_i^-}) \phi_i(\omega),
\]
where for each $r$ the sum is $L^2(\Sigma)$ orthogonal.

\subsection{Hardt-Simon foliation}\label{sec:hs}

Taking $\bC$, $\Sigma$ as above, then $\bC$ divides $\R^{n+1}$ into two connected, open, disjoint regions $E_+$ and $E_-$.  We can choose an oriented unit normal $\nu_\bC$ for $\bC$, so that $\nu_\bC$ points into $E_+$.

When $\bC$ is area-minimizing, in the sense of currents, then \cite{HaSi} have shown there are \emph{smooth}, area-minimizing hypersurfaces $S_\pm \subset E_\pm$, which are asymptotic to $\bC$.  Moreover, the $S_\pm$ radial graphs, and hence the collection of dilations $\lambda S_\pm$ ($\lambda > 0$) forms a foliation of $E_\pm$ by smooth, area-minimizing hypersurfaces, sometimes called the Hardt-Simon foliation.  Let us orient $S_\pm$ with unit normals $\nu_{S_{\pm}}$ compatible with $\bC$, so that as $|x| \to \infty$, $\nu_{S_{\pm}} \to \nu_\bC$.

When $\bC$ is \emph{strictly} minimizing, then $S_\pm$ decays to $\bC$ like the larger homogeneity $r^{\gamma_1^+}$.  In particular, after a normalization as necessary, there is a radius $R_0 \geq 1$ and $\alpha_0 > 0$ so that
\[
S_\pm  \setminus B_{R_0} = \graph_\bC(v_\pm),
\]
where $v_\pm : \bC \setminus B_{R_0/2} \to \R$ is a smooth function satisfying
\[
v_\pm(r\omega) = \pm r^{\gamma^+_1} f_\pm(r\omega), \quad \sum_{k=0}^2 r^k|\nabla^k (f_\pm - \phi_1)| = O(r^{-\alpha_0}).
\]
For shorthand we will set $\gamma = \gamma_1^+$.  See \cite{HaSi} for details about strictly minimizing.

Given $\lambda \in \R$, define
\[
S_\lambda = \left\{ \begin{array}{l l} \lambda S_+ & \lambda > 0 \\ \bC & \lambda = 0 \\ |\lambda| S_- & \lambda < 0 \end{array} \right. \qquad v_\lambda(r \omega) = \left\{\begin{array}{l l} \lambda v_+(r/\lambda) & \lambda > 0 \\ 0 & \lambda = 0 \\ |\lambda| v_-(r/|\lambda|) & \lambda < 0 \end{array} \right.
\]
so that
\[
S_\lambda \setminus B_{\lambda R_0} = \graph_\bC(v_\lambda).
\]
Let $S_\lambda$ have the same orientation as $S_{\sign(\lambda)}$.  Observe that
\begin{equation}\label{eqn:form-of-v}
v_\lambda(r) = \sign(\lambda) |\lambda|^{1-\gamma} r^\gamma f_{\sign(\lambda)}(r/|\lambda|).
\end{equation}
For shorthand, we will often write $\lambda^{\alpha} := \sign(\lambda) |\lambda|^\alpha$.

The following straightforward Lemma will be useful. 
\begin{lemma}\label{lem:v-diff}
	For $|\mu|, |\lambda| \leq 1$, and $r \geq R \geq R_0$, we have
	\[
	v_\mu(r) - v_\lambda(r) = (1+O( (\max\{|\mu|, |\lambda|\})^{\alpha_0} R^{-{\alpha_0}})) (\mu^{1-\gamma} - \lambda^{1-\gamma})r^{\gamma}.
	\]
	In particular, if $R_0(\bC)$ is sufficiently large, then
	\[
	\frac{1}{4} (\mu^{1-\gamma} - \lambda^{1-\gamma})^2 r^{2\gamma} \leq (v_\mu(r) - v_\lambda(r))^2 \leq 4(\mu^{1-\gamma} - \lambda^{1-\gamma})^2 r^{2\gamma}
	\]
	for all $r \geq \max\{|\mu|, |\lambda|\} R_0$.
\end{lemma}

\begin{proof}
	If $\lambda \neq 0$, $|\lambda| \leq 1$, and $r \geq |\lambda| R$, then we have
	\begin{align*}
	\frac{d}{d\lambda} v_\lambda(r) 
	&= (1-\gamma)|\lambda|^{-\gamma} r^{\gamma} f_{\sign(\lambda)}(r/|\lambda|) - |\lambda|^{-1-\gamma} r^{1+\gamma} f_{\sign(\lambda)}' (r/|\lambda|) \\
	&= (1-\gamma)|\lambda|^{-\gamma} r^\gamma (1 + O(|\lambda|^{\alpha_0} R^{-\alpha_0})) .
	\end{align*}
	If $\mu \lambda \geq 0$, then the required result follows from the above and the fundamental theorem of calculus.
	
	If $\mu\lambda < 0$, $|\lambda| \geq |\mu| > 0$, then we have (recalling our shorthand $\mu^\beta = \sign(\mu)|\mu|^\beta$)
	\begin{align*}
	(v_\mu(r) - v_\lambda(r))^2 
	&= (|\mu|^{1-\gamma} + |\lambda|^{1-\gamma})^2 r^{2\gamma} (1+O(|\lambda|^{\alpha_0} R^{-\alpha_0})) \\
	&= (\mu^{1-\gamma} - \lambda^{1-\gamma})^2 r^{2\gamma} (1+O(|\lambda|^{\alpha_0} R^{-\alpha_0})) .
	\end{align*}
\end{proof}

\subsection{Minimizing quadratic cones}

Take $\bC = \bC^{p,q}$ an area-minimizing quadratic cone.  There are two key properties of $\bC$ which we shall need.  These are proven in \cite[Proposition 2.7]{SiSo}.

\begin{enumerate}
\item $\bC$ is \emph{strictly-minimizing}, so that the foliation decays like $r^{\gamma_1^+}$.  In fact, since $\bC$ is rotationally symmetric, we have that 
\begin{equation}\label{eqn:rot-sym}
v_\pm(r\omega) = \pm r^{\gamma^+_1} f_\pm(r\omega), \quad \sum_{k=0}^2 r^k |\nabla^k (f_\pm - 1)| = O(r^{-\alpha_0}).
\end{equation}
Recall that for shorthand we write $\gamma = \gamma_1^+$.

\item $\bC$ is \emph{strongly integrable}, in the following sense: any solution of $\cL_\bC(w) = 0$ can be written
\[
w(x = r\omega) = \sum_{j \geq 1} a_i^- r^{\gamma_i^-} \phi_i(\omega) + e r^{\gamma_1^+} + (b + Ax) \cdot \nu_\bC(\omega) + \sum_{j \geq 4} a_i^+ r^{\gamma_i^+},
\]
where $e \in \R$, $b \in \R^{n+1}$, and $A$ is a skew-symmetric $(n+1)\times (n+1)$ matrix.  In other words, $\gamma_2^+ = 0$, $\gamma_3^+ = 1$, and the eigenfunctions $\phi_2$, $\phi_3$ are generated by translations, rotations.
\end{enumerate}

Every result in our paper holds for any area-minimizing hypercone satisfying the above two conditions.  Rotational symmetry as in \eqref{eqn:rot-sym} simplifies our computations slightly, but has no bearing on our proof.  We write all our results for quadratic cones because these are the only area-minimizing cones which we can verify as ``strongly integrable.''

\section{Main theorems}\label{sec:main}

\emph{For the duration of this paper}, we fix $\bC = \bC^{p,q}$ to be an area-minimizing quadratic cone, $S_\lambda$ the Hardt-Simon foliation, and we use the notation associated to $\bC$, $S_\lambda$ as introduced in Section \ref{sec:hs}.

Our main theorem is the following, from which Theorem \ref{thm:teaser} follows directly.


\begin{theorem}\label{thm:main}
There are constants $\delta_{1}(\bC)$, $\Lambda_{1}(\bC)$, $c_{1}(\bC)$, $\beta(\bC)$ so that the following holds.  Take $|\lambda| \leq \Lambda_{1}$, and let $M$ be a stationary integral varifold in $B_1$, satisfying
\begin{gather}\label{eqn:main-hyp}
\begin{aligned}
& \quad\quad\quad \quad \quad\quad \quad\int_{B_1} d_{S_\lambda}^2 d\mu_M \leq E \leq \delta_{1}^2,\\ 
&\mu_M(B_1) \leq (3/2) \mu_\bC(B_1), \quad\mbox{and}\quad \mu_M(\overline{B_{1/10}}) \geq (1/2) \mu_\bC(B_{1/10}).
\end{aligned}
\end{gather}
Then there is an $a \in \R^{n+1}$, $q \in SO(n+1)$, $\lambda' \in \R$, with
\begin{equation}\label{eqn:main-concl1}
|a| + |q - Id| + \left|\sign(\lambda')|\lambda'|^{1-\gamma} - \sign(\lambda)|\lambda|^{1-\gamma}\right| \leq c_{1} E^{1/2},
\end{equation}
and a $C^{1,\beta}$ function $u : (a + q(S_{\lambda'})) \cap B_{1/2} \to \R$, so that 
\[
\spt M \cap B_{1/2} = \graph_{a + q(S_{\lambda'})}(u) \cap B_{1/2},
\] 
and  $u$ satisfies the estimates
\begin{equation}\label{eqn:main-concl2}
r^{-1} |u|_{C^0(B_r(a))} + |\nabla u|_{C^0(B_r(a))} + r^{\beta} [\nabla u]_{\beta, B_r(a)} \leq c_{1} r^{\beta} E^{1/2} \quad \forall r \leq 1/2.
\end{equation}

In particular, $M \cap B_{1/2}$ is either smooth, or has an isolated singularity modeled on $\bC$.
\end{theorem}

\begin{remark}
The precise form of the lower bound on $\mu_M(B_{1/10})$ in \eqref{eqn:main-hyp} is of no consequence, nor is the precise ball radius $1/10$.  One could easily assume (for example) $\mu_M(B_{1/10}) \geq v > 0$, and obtain the same conclusions, except that the constants $\delta_{1}$ and $\Lambda_{1}$ would depend on the choice of $v$ also.  The upper bound on $\mu_M(B_1)$ is more important: we require it to be strictly less than $2 \mu_\bC(B_1)$.
\end{remark}

A further characterization is possible in the case when $M$ as in Theorem \ref{thm:main} is singular.  \cite{CaHaSi} have constructed a large class of examples of minimal surfaces in $B_1$, which are singular perturbations of a given minimal cone (see Section \ref{sec:uniqueness}).  In fact, in a sufficiently small neighborhood, these are the only minimal surfaces which are graphical over $\bC$.  It would be interesting to know whether examples like those in \cite{CaHaSi} exist as perturbations over a foliate $S_\lambda$.
\begin{prop}\label{prop:uniqueness}
Let $M$, $\lambda'$ be as in Theorem \ref{thm:main}.  If $\lambda' = 0$, and $E \leq \delta_{2}(\bC)$ is sufficiently small, then $\spt M \cap B_{1/4}$ coincides with one of the graphical solutions as constructed by \cite{CaHaSi}.
\end{prop}

The most interesting consequence of Theorem \ref{thm:main} is that singularities modeled on (minimizing) Simons's cones propagate out their structure not only to a neighborhood of the original surface, but also of \emph{nearby} surfaces.  If these nearby surfaces are not minimal, but instead of $L^p$ mean curvature, then essentially the same structure holds, but with slightly less regularity.  In this sense the minimizing Simons's singularities can be thought of as ``very strongly isolated.''

\begin{corollary}\label{cor:mc}
Given any $p > n$, there are constants $\delta_{3}(\bC)$, $\eps_{3}(p, \bC)$, $\Lambda_{3}(\bC)$, $c_{3}(\bC)$ so that the following holds.  Let $M$ be an integral $n$-varifold in $B_1$ with generalized mean curvature $H_M$, zero generalized boundary, satisfying \ref{eqn:main-hyp} with $\delta_{3}^2$ in place of $E$, and 
\[
\left( \int_{B_1} |H_M|^p d\mu_M \right)^{1/p} \leq \eps_{3}.
\]

Then there are $a \in \R^{n+1}$, $\lambda' \in \R$ so that 
\begin{itemize}
	\item either $\lambda' = 0$, in which case $\spt M \cap B_{1/2}$ is a $C^{1,\beta}$ perturbation of $\bC$;
	\item or $\lambda' \neq 0$, and we can find for every $0 < r \leq 1/2$ a $q_r \in SO(n+1)$, so that $\spt M \cap B_r(a) \setminus B_{r/100}(a)$ is a $C^{1,\beta}$ graph over $a + q_r(S_{\lambda'})$. 
\end{itemize}  
In particular, $M$ is either entirely regular, or has an isolated singularity modeled on $\bC$. 
\end{corollary}

\begin{example}
This Corollary rules out many possible examples of singularity formation.  For example, in an $8$ dimensional manifold this rules out the possibility that $S^3\times S^3$ singularites are collapsing into an $S^2 \times S^4$ singularity, or even worse, that multiple types of isolated singularities are collapsing into and single $S^3\times S^3$ or $S^2 \times S^4$ singularity.
\end{example}

\begin{remark}
We cannot obtain directly that the $q_r$ have a limit as $r \to 0$.  If $\lambda' = 0$, then we can use \cite{AllAlm} to deduce a posteriori that $M \cap B_{1/2}$ is a $C^{1,\alpha}$ perturbation of $\bC$.  If $\lambda' \neq 0$, then we do not need to worry about the limiting behavior of the $q_r$ to deduce that $M \cap B_{1/2}$ is some $C^{1,\alpha}$ perturbation of $S_{\lambda'}$.  However, because we have no control over $q_r$ as $r \to 0$, we cannot obtain any effective estimates on the $C^{1,\alpha}$ map in question.  It would be interesting to resolve this.

\end{remark}

Another direct consequence of our regularity theorem is the following rigidity theorem for area-minimizing surfaces asymptotic to Simons's cones, which was originally proven by Simon-Solomon:

\begin{corollary}[\cite{SiSo}]\label{cor:ss}
Let $M$ be an area-minimizing hypersurface in $\R^{n+1}$, and $\bC$ be a strongly integrable cone, with associated foliation $S_\lambda$.  Suppose there is a sequence of radii $R_i \to \infty$ so that
\[
R_i^{-1} M \to \bC 
\]
in the flat disance.  Then up to translation, rotation, and dilation, $M = \bC, S_{1}, $ or $S_{-1}$.
\end{corollary}

\begin{remark}
We remark that our result is much stronger than the characterization of \cite{SiSo}.  As illustrated by the examples of \cite{CaHaSi}, being close to the Simons's cone at scale $1$ is much weaker than being close on all of $\R^{n+1}$ -- in particular, the latter precludes any modes growing faster than $1$-homogeneous.  One can think of \cite{SiSo} as a Bernstein-type theorem, while our theorem as an Allard-type regularity theorem.
\end{remark}

\section{Outline of proof}

Our strategy is to prove the following excess-decay type theorem (Proposition \ref{prop:decay}): provided both $\lambda$ and $\int_{B_1} d_{S_\lambda}^2 d\mu_M$ are sufficiently small (plus some restrictions on the mass of $M$), then we have a decay estimate of the form
\begin{equation}\label{eqn:outline-decay}
\theta^{-n-2} \int_{B_\theta} d_{a' + q'(S_{\lambda'})}^2 d\mu_M \leq (1/2) \int_{B_1} d_{S_\lambda}^2 d\mu_M.
\end{equation}
That is, provided $S_\lambda$ is sufficiently close to the cone $\bC$, and we are sufficiently $L^2$ close to $S_\lambda$, then after a translation/rotation/dilation as necessary, our $L^2$ distance to the foliation improves at a smaller scale.

We can continue iterating \eqref{eqn:outline-decay} while $S_\lambda$ is scale-invariantly close to $\bC$, and obtain a decay of the form: there is an $a'' \in \R^{n+1}$, $q'' \in SO(n+1)$, and $\lambda'' \in \R$ so that
\begin{equation}\label{eqn:outline-decay2}
r^{-n-2} \int_{B_r} d_{a'' + q''(S_{\lambda''})}^2 d\mu_M \leq C r^{2\beta} \quad \forall c |\lambda''| \leq r \leq 1.
\end{equation}
By two straightforward contradiction arguments (one for $M$ close to $S_\lambda \cap B_1 \setminus B_{1/100}$ with $\lambda$ small, and one for $M$ close to $S_{\pm 1} \cap B_{c}$), we can use \cite{All} with \eqref{eqn:outline-decay2} to deduce $\spt M \cap B_{1/2}$ is graphical over $a'' + q''(S_{\lambda''})$.

\vspace{5mm}

We would like to prove \eqref{eqn:outline-decay} by contradiction,with an argument that loosely resembles the original ``excess decay'' proof due to De Giorgi and, as implemented in a fashion closer to our style, \cite{AllAlm}, \cite{Simon1}, \cite{Simon-CAG}.  Briefly, we would like to suppose \eqref{eqn:outline-decay} fails for some sequence $M_i$ and $\lambda_i \to 0$, $E_i = \int_{B_1} d_{S_{\lambda_i}}^2 d\mu_{M_i} \to 0$.  Then over larger and larger annuli $B_{1/2} \setminus B_{\tau_i}$ ($\tau_i \to 0$) we can write $\spt M = \graph_{S_{\lambda_i}}(u_i)$.  If we rescale $v_i = E_i^{-1/2} u_i$, then the $v_i$ have uniformly bounded $||v_i||_{L^2(B_1)}$, and after passing to a subsequence we get convergence
\[
v_i \to w\quad \mbox{with} \quad \cL_\bC(w) = 0\,.
\]
The idea now, in vague terms, is to use good decay properties for solutions to the linearized problem $\cL_\bC(w) = 0$, to prove good decay for solutions to the non-linear problem $\cM_\bC(u) = 0$, that is we'd like to arrange so that $w = O(r^{1+\eps})$, and then use this to deduce $L^2$ decay of the $u_i$ as in \eqref{eqn:outline-decay}.

To ensure this argument works we need to 
\begin{itemize}
	\item[(i)] ensure the decaying norm for the non-linear problem is comparable to the linear one (a.k.a. non-concentration of $L^2$ norm at singularities), that is for any $\rho$ small
	\[
	E_i^{-1} \int_{B_\rho} d_{S_{\lambda_i}}^2 d\mu_{M_i} \to ||w||_{L^2(\bC \cap B_\rho)}^2\,,
	\]
	\item[(ii)] prove good decay for the linear problem (a.k.a. killing bad homogeneities through integrability), that is $w = O(r^{1+\eps})$. 
\end{itemize}

The latter issue is where the concept of \emph{integrability} arises.  A minimal cone $\bC$ is called integrable if every $1$-homogeneous Jacobi field arises from a $1$-parameter family of minimal cones.  The idea of \cite{All}, \cite{AllAlm} is that, under suitable density assumptions in the argument above, one can typically show that $w=O(r)$, but one needs it to grow like $r^{1+\eps}$.  For a cone with an isolated singularity, the homogeneities are discrete, and so provided $\bC$ is integrable, one can rewrite the minimal surface as a graph over a slightly adjusted cone, chosen to cancel the $r$ term in the Fourier expansion of $w$.  

The main novelty in our approach is in treating the foliation $S_\lambda$ as a direction of integrability.  In other words, we are relaxing the original notion of integrability, as a movement through cones, to allow one to push off the cone into families of entirely smooth hypersurfaces, and in particular we are allowing for a notion of integrability in which the singularity behavior changes.  In order to handle this we require new decay and non-concentration estimates for minimal surfaces near an arbitrary foliate $S_\lambda$ without any structural assumption on $M$.  This is the content of Theorem \ref{thm:non-conc}.  

More precisely, the key observation is that the foliation is generated by a positive Jacobi field of the form
\[
v(r) = r^{\gamma}, \quad 0 > \gamma > - (n-2)/2,
\]
and this Jacobi field has itself good $L^2$ decay:
\begin{equation}\label{eqn:outline-jacobi-decay}
\int_{B_\rho \cap \bC} v^2 d\haus^n \leq c \rho^2 \int_{B_1 \cap \bC} v^2 d\haus^n.
\end{equation}

To deal with point (i), we use the maximum principle to ``trap'' $M$ between two foliates, and thereby show that $M$ cannot diverge from a given $S_\lambda$ any faster than the foliation itself.  This allows us to prove that $\int_{B_\rho} d_{S_\lambda}^2 d\mu_M$ has a decay similar to \eqref{eqn:outline-jacobi-decay}, and hence no $L^2$ norm can accumulate near the non-graphical region (away from $0$ we of course have strong $L^2$ convergence since the $v_i$ converge smoothly there).

To deal with point (ii), we can prove that the $v_i$, and hence the resulting Jacobi field $w$, grow at least as fast as $v(r) = r^\gamma$ as $r$ increases.  Using the strongly integrable nature of $\bC$, we can then deduce that $w$ looks like
\[
w(x = r\omega) = e r^\gamma + (b + Ax)\cdot \nu_\bC + O(r^{1+\eps}) ,
\]
for $f \in \R$, $b \in \R^{n+1}$, $A$ skew-symmetric.  In other words, $w$ has the growth we require except for terms generated by moving into the foliation, translation, and rotation. By replacing the $S_{\lambda_i}$ with a new sequence of foliates $a_i + q_i(S_{\lambda_i'})$, and repeating the above contradiction argument with this new sequence, we can arrange so that these three lower homogeneities disappear, and thereby deduce $w = O(r^{1+\eps})$.

\end{comment}

\section{Non-concentration of $L^2$-Excess}

Our main theorem of this section is the following.  It is spiritually similar to Theorem 2.1 in \cite{Simon-CAG}, except we are proving non-concentration with respect to an arbitrary foliate $S_\lambda$ instead of just $\bC$, and we additionally obtain a pointwise decay estimate on the graphing function.  Recall the shorthand $\gamma = \gamma_1^+$.

\begin{theorem}\label{thm:non-conc}
For every $0 < \tau < 1/4$, $\beta > 0$, there is $\Lambda_{4}(\bC, \tau)$, $\eps_{4}(\bC, \beta, \tau)$, $c_{4}(\bC)$ so that the following holds: if $|\lambda| \leq \Lambda_{4}$ and $M$ is a stationary integral $n$-varifold in $B_1$ satisfying
\begin{equation}\label{eqn:non-conc-hyp}
\int_{B_1} d_{S_\lambda}^2 d\mu_M \leq \eps_{4}^2, \quad \mu_M(B_1) \leq (7/4) \mu_\bC(B_1), \quad \mu_M(\overline{B_{1/10}}) \geq \frac{1}{2} \mu_\bC(B_{1/10}),
\end{equation}
then there is a smooth function $u : S_\lambda \cap B_{1/2} \setminus B_{\tau/2} \to \R$ so that
\begin{equation}\label{eqn:non-conc-concl1}
\spt M \cap B_{1/2} \setminus B_\tau = \graph_{S_\lambda}(u) \cap B_{1/2} \setminus B_\tau, \quad \sum_{k=0}^3 r^{k-1} |\nabla^k u| \leq \beta .
\end{equation}

For every $\tau \leq \rho \leq 1/4$, we have:
\begin{align}
\int_{B_\rho \setminus B_\tau} u^2 d\mu_\bC + \int_{B_\rho} d_{S_\lambda}^2 d\mu_M 
&\leq c_{4} \rho^{n+2\gamma} \int_{B_1} d_{S_{\lambda}}^2 d\mu_M \label{eqn:non-conc-concl2} \\
&\leq c_{4} \rho^2 \int_{B_1} d_{S_{\lambda}}^2 d\mu_M \nonumber
\end{align}
Moreover, $u$ has the following $L^\infty$ decay bound:
\begin{equation}\label{eqn:non-conc-concl3}
\sup_{S_\lambda \cap  \del B_r} u^2 \leq c_{4} r^{2\gamma} \int_{B_1} d_{S_\lambda}^2 \qquad \forall r \in (\tau, 1/4).
\end{equation}
\end{theorem}

\begin{proof}
Let $R_0(\bC)$ be as in Lemma \ref{lem:v-diff}.  Taking $\Lambda_{4}(\bC, \tau)$ sufficiently small, we can assume $\Lambda_{4} R_0 < \tau/100$, and
\begin{equation}\label{eqn:v-bounds}
\sum_{k=0}^2 r^{k-1} |\nabla^k v_\lambda| \leq \beta \quad \forall r \in (\tau/100, 1).
\end{equation}
By a straightforward contradiction argument, allowing $|\lambda| \leq \Lambda_{4}$ to vary, we get that (taking $\eps_{4}(\bC, \Lambda_{4}, \tau, \beta)$ small):
\begin{equation}\label{eqn:M-bounds}
\spt M \cap B_{3/4} \setminus B_{\tau/10} = \graph_{S_\lambda}(u), \quad \sum_{k=0}^3 r^{k-1} |\nabla^k u| \leq \beta.
\end{equation}

Indeed, otherwise there is a sequence of stationary integral $n$-varifolds $M_i$, and numbers $\eps_i \to 0$, $\lambda_i \in [-\Lambda_{4}, \Lambda_{4}]$, for which  \eqref{eqn:non-conc-hyp} holds but \eqref{eqn:M-bounds} fails.  We can without loss assume $\lambda_i \to \lambda$, for some $|\lambda| \leq \Lambda_{4}$.  By compactness of stationary varifolds with bounded mass, we can pass to a subsequence (also denoted $i$), and get varifold convergence $M_i \to M$ for some stationary integral $n$-varifold $M$.  The resulting $M$ satisfies
\begin{equation}\label{eqn:M_i-limit-cond}
\int_{B_1} d_{S_\lambda}^2 d\mu_M = 0, \quad \mu_M(B_1) \leq (7/4) \mu_\bC(B_1), \quad \mu_M(B_{2/10}) \geq (1/2) \mu_\bC(B_{1/10}).
\end{equation}
The constancy theorem implies that $M = k [S_\lambda]$, for some integer $k$.  The lower bound of \eqref{eqn:M_i-limit-cond} implies $k \geq 1$.  Ensuring $\Lambda_{4}(\bC)$ is sufficiently small, the upper bound in \eqref{eqn:M_i-limit-cond} implies $k \leq 1$.  So in fact $M_i \to [S_\lambda]$, and hence by Allard's theorem convergence is smooth on compact subsets of $B_1 \setminus \{0\}$.  This proves our assertion.

It will be more convenient in this proof to work with graphs over $\bC$.  By a similar contradiction argument as above, we have (again taking $\eps_{4}(\bC, \tau, \beta)$, $\Lambda_4(\bC, \tau, \beta)$ sufficiently small):
\begin{equation}\label{eqn:h-bounds}
\spt M \cap B_{3/4} \setminus B_{\tau/10} = \graph_\bC(h), \quad \sum_{k=0}^3 r^{k-1} |\nabla^k h| \leq \beta.
\end{equation}

Ensuring $\beta(\bC, \tau)$ is sufficiently small, $u$ is effectively equivalent to $h - v_\lambda$.  Precisely, if $r\omega \in \bC$, and $x = r\omega + v_\lambda(r) \nu_\bC(r\omega) \in S_\lambda$, then an elementary computation shows that
\[
|u(x)\nu_{S_\lambda}(x) - (h(r\omega) - v_\lambda(r))\nu_\bC(r\omega)| \leq \Psi(\beta | \bC, \tau) |h(r\omega) - v_\lambda(r)|.
\]
Choosing a possibly small $\beta(\bC, \tau)$, this implies
\[
\sup_{S_\lambda \cap \del B_r} |u| \leq 2\sup_{\bC \cap B_{2r} \setminus B_{r/2}} |h - v_\lambda| \qquad \forall r \in (\tau/2, 1/4)\,,
\]
and
\begin{align*}
\int_{B_r \setminus B_{\tau/2}} u^2 d\mu_{S_\lambda} \leq 2 \int_{B_{2r} \setminus B_{\tau/4}} |h - v_\lambda|^2 d\mu_\bC \\
\int_{B_r \setminus B_{\tau/2}} |h - v_\lambda|^2 d\mu_\bC \leq 2 \int_{B_{2r} \setminus B_{\tau/4}} d_{S_\lambda}^2 d\mu_M \\
\int_{B_r \setminus B_{\tau/2}} d_{S_\lambda}^2 d\mu_M \leq 2 \int_{B_{2r} \setminus B_{\tau/4}} u^2 d\mu_{S_\lambda}.
\end{align*}
So we can prove the required estimates for $h - v_\lambda$ instead of $u$.

For $\rho \in (\tau/10, 3/4)$, define
\begin{align*}
\lambda_\rho^+ = \inf \{ \mu : v_\mu(\rho\omega) \geq h(\rho\omega) \quad \forall \omega \in \Sigma \} \\
\lambda_\rho^- = \sup \{ \mu : v_\mu(\rho\omega) \leq h(\rho\omega) \quad \forall \omega \in \Sigma \} .
\end{align*}
From \eqref{eqn:form-of-v}, \eqref{eqn:h-bounds}, we have $|\lambda_\rho^\pm| = \Psi(\beta | \bC, \tau)$, and so ensuring $\beta(\bC, \tau)$ is sufficiently small, $\lambda^\pm_\rho R_0 < \tau/10$ for all admissible $\rho$.  By the maximum principle (e.g. \cite{SoWh}), we have that $\spt M \cap B_\rho$ is trapped between $S_{\lambda_\rho^-}$ and $S_{\lambda_\rho^+}$.  This implies that $\lambda_\rho^+$ is increasing in $\rho$, while $\lambda_\rho^-$ is decreasing in $\rho$, and
\begin{equation}\label{eqn:v-bound}
v_{\lambda^-\rho}(r) \leq h(r\omega) \leq v_{\lambda^+_\rho}(r) \quad \forall r \in (\tau/2, \rho)
\end{equation}

Both $h$ and $v_\lambda$ solve the minimal surface equation over $\bC$, and on any compact subset of $\bC \cap B_1 \setminus \{0\}$ have uniformly bounded derivatives.  Therefore the difference $h - v_\lambda$ solves a linear, second order, uniformly elliptic operator.  So by standard iteration techniques at scale $r$, we get
\begin{align}\label{eqn:H-diff-bound}
\sup_\omega |h(r\omega) - v_\lambda(r)|^2 \leq c(\bC) r^{-n} \int_{B_{2r} \setminus B_{r/2}} |h - v_\lambda|^2 d\mu_\bC \quad \forall \tau/5 < r < 1/4 .
\end{align}

Since $\lambda_\rho^+$ is increasing in $\rho$, $\lambda_\rho^-$ is decreasing in $\rho$, and $\lambda^-_\rho \leq \lambda^+_\rho$, we get that
\begin{equation}\label{eqn:max-incr}
\max \{ ((\lambda^+_\rho)^{1-\gamma} - \lambda^{1-\gamma})^2, ((\lambda^-_\rho)^{1-\gamma} - \lambda^{1-\gamma})^2 \}
\quad\mbox{is increasing in $\rho$}.
\end{equation}

For any $\tau/5 < r < \rho < 1/4$, we have by Lemma \ref{lem:v-diff}, \eqref{eqn:v-bound}, \eqref{eqn:max-incr}, and \eqref{eqn:H-diff-bound}:
\begin{align}
|h(r\omega) - v_\lambda(r)|^2
&\leq 2 \max \{ (v_{\lambda^+_\rho}(r) - v_\lambda(r))^2, (v_{\lambda^-_\rho}(r) - v_\lambda(r))^2 \} \nonumber \\
&\leq c r^{2\gamma} \max \{ ((\lambda^+_\rho)^{1-\gamma} - \lambda^{1-\gamma})^2, ((\lambda^-_\rho)^{1-\gamma} - \lambda^{1-\gamma})^2 \} \nonumber \\
&\leq c r^{2\gamma} \max \{ ((\lambda^+_{1/4})^{1-\gamma} - \lambda^{1-\gamma})^2, ((\lambda^-_{1/4})^{1-\gamma} - \lambda^{1-\gamma})^2 \} \nonumber\\
&\leq c r^{2\gamma} \max \{ (v_{\lambda^+_{1/4}}(1/4) - v_\lambda(1/4))^2, (v_{\lambda^-_{1/4}}(1/4) - v_\lambda(1/4))^2 \} \nonumber\\
&= c r^{2\gamma} \max\{ (\sup_\omega h(\omega/4) - v_\lambda(1/4))^2, (\inf_\omega h(\omega/4) - v_\lambda(1/4))^2 \} \nonumber\\
&= c r^{2\gamma}  \sup_\omega |h(\omega/4) - v_\lambda(1/4)|^2 \nonumber\\
&\leq c r^{2\gamma} \int_{B_{1/2}\setminus B_{1/8}} |h - v_\lambda|^2 d\mu_\bC, \label{eqn:h-v-l2-bounds}
\end{align}
where $c = c(\bC)$.  This implies the estimate \eqref{eqn:non-conc-concl3}.

Integrating this relation in $r \in (\tau/2, \rho)$, gives:
\begin{align*}
\int_{B_\rho \setminus B_{\tau/2}} |h - v_\lambda|^2 d\mu_\bC 
&\leq \frac{c(\bC)}{n+2\gamma} \rho^{n+2\gamma} \int_{B_1} d_{S_{\lambda}}^2 d\mu_M \\
&\leq c(\bC) \rho^2 \int_{B_1} d_{S_\lambda}^2 d\mu_M, 
\end{align*}
since $2\gamma + n \geq -(n-2) + n \geq 2$.  By the bounds \eqref{eqn:v-bounds}, \eqref{eqn:h-bounds}, we have
\begin{equation}\label{eqn:annulus-l2}
\int_{B_\rho \setminus B_\tau} d_{S_\lambda}^2 d\mu_M \leq c \rho^2 \int_{B_1} d_{S_\lambda}^2 d\mu_M.
\end{equation}

We focus now on proving the final part of \eqref{eqn:non-conc-concl2}, i.e. the excess in the ball $B_\tau$.  Since $S$ is graphical over $\bC$ near $\del B_{R_0}$, we have
\[
d_H(S_\lambda \cap B_{\max\{|\lambda|, |\mu|\}R_0}, S_\mu \cap B_{\max\{|\lambda|, |\mu|\}R_0}) \leq c(\bC) |\lambda - \mu|.
\]
Let $\lambda_\tau = \max\{|\lambda^+_\tau|, |\lambda^-_\tau|\}$.  Then, recalling how $\spt M \cap B_\tau$ is trapped between $S_{\lambda^+_\tau}$ and $S_{\lambda^-_\tau}$, we get
\begin{align}
&\int_{B_{\lambda_\tau R_0}} d_{S_\lambda}^2 d\mu_M \nonumber\\
&\leq \max\{ d_H(S_{\lambda^+_\tau} \cap B_{\lambda_\tau R_0}, S_{\lambda} \cap B_{\lambda_\tau R_0})^2, d_H(S_{\lambda^-_\tau} \cap B_{\lambda_\tau R_0}, S_\lambda \cap B_{\lambda_\tau R_0})^2\} \mu_M(B_{\lambda_\tau R_0}) \nonumber\\
&\leq c(\bC) \lambda_\tau^n \max\{ (\lambda_\tau^+ - \lambda)^2, (\lambda_\tau^- - \lambda)^2 \} \nonumber \\
&\leq c(\bC) \tau^{n+2\gamma}  \max \{ ((\lambda^+_\tau)^{1-\gamma} - \lambda^{1-\gamma})^2, ((\lambda^-_\tau)^{1-\gamma} - \lambda^{1-\gamma})^2 \} . \label{eqn:l2-inside-lambda}
\end{align}
The last line follows because there is a constant $c(n)$ so that whenever $|\mu|, |\lambda| \leq 1$, we have (recall $\gamma < 0$)
\[
(\mu^{1-\gamma} - \lambda^{1-\gamma})^2 \geq \frac{1}{c(n)} \max \{ |\mu|, |\lambda|\}^{-2\gamma} (\mu - \lambda)^2 .
\]

Choose $I$ so that $2^I \lambda_\tau R_0 \leq \tau < 2^{I+1} \lambda_\tau R_0$.  We compute:
\begin{align}
&\int_{B_\tau \setminus B_{\lambda_\tau R_0}} d_{S_\lambda}^2 d\mu_M \nonumber\\
&\leq \sum_{i=0}^{I} \int_{B_{2^{i+1} \lambda_\tau R_0} \setminus B_{2^i \lambda_\tau R_0}} d_{S_\lambda}^2 d\mu_M \nonumber\\
&\leq \sum_{i=0}^I \sup_{ 2^i \lambda_\tau R_0 \leq r \leq 2^{i+1} \lambda_\tau R_0} \max\{ (v_{\lambda_\tau^+}(r) - v_\lambda(r))^2, (v_{\lambda^-_\tau}(r) - v_\lambda(r))^2 \} \mu_M(B_{2^{i+1}\lambda_\tau R_0}) \nonumber\\
&\leq \sum_{i=0}^I c(\bC) (2^i \lambda_\tau R_0)^{2\gamma} \max \{ ((\lambda^+_\tau)^{1-\gamma} - \lambda^{1-\gamma})^2, ((\lambda^-_\tau)^{1-\gamma} - \lambda^{1-\gamma})^2 \} (2^i \lambda_\tau R_0)^n \nonumber\\
&\leq c(\bC) \tau^{n+2\gamma} \max \{ ((\lambda^+_\tau)^{1-\gamma} - \lambda^{1-\gamma})^2, ((\lambda^-_\tau)^{1-\gamma} - \lambda^{1-\gamma})^2 \} . \label{eqn:l2-annulus}
\end{align}

Combining \eqref{eqn:l2-inside-lambda}, \eqref{eqn:l2-annulus}, with the computations of \eqref{eqn:h-v-l2-bounds}, we obtain
\begin{align*}
\int_{B_\tau} d_{S_\lambda}^2  d\mu_M
&\leq c(\bC) \tau^{n+2\gamma} \max \{ ((\lambda^+_\tau)^{1-\gamma} - \lambda^{1-\gamma})^2, ((\lambda^-_\tau)^{1-\gamma} - \lambda^{1-\gamma})^2 \} \\
&\leq c(\bC) \tau^2 \int_{B_1} d_{S_\lambda}^2 d\mu_M.
\end{align*}
Together with \eqref{eqn:annulus-l2}, this gives the required estimate \eqref{eqn:non-conc-concl2}.
\end{proof}

The following corollary will also be useful.
\begin{corollary}\label{cor:change-lambda}
There is a $\Lambda_{5}(\bC)$ so that if $|\lambda|, |\lambda'| \leq \Lambda_{5}$, and $M$ satisfies the hypotheses of Theorem \ref{thm:non-conc}, then
\[
\int_{B_1} d_{S_{\lambda'}}^2 d\mu_M \leq c(\bC) \int_{B_1} d_{S_{\lambda}}^2 d\mu_M + c(\bC) ( \lambda^{1-\gamma} - (\lambda')^{1-\gamma})^2.
\]
\end{corollary}

\begin{proof}
The computations of Theorem \ref{thm:non-conc} show that, provided $|\lambda|, |\lambda'| \leq \Lambda_{4}(\bC)$, we have
\begin{align*}
\int_{B_{1/4}} d_{S_{\lambda'}}^2 d\mu_M 
&\leq c(\bC) \max\{ ((\lambda^+_{1/4})^{1-\gamma} - (\lambda')^{1-\gamma})^2, ((\lambda_{1/4}^-)^{1-\gamma} - (\lambda')^{1-\gamma})^2 \} \\
&\leq c(\bC) \int_{B_1} d_{S_\lambda}^2 d\mu_M + c(\bC) ( \lambda^{1-\gamma} - (\lambda')^{1-\gamma})^2.
\end{align*}
It remains only to control the annuluar region $B_1 \setminus B_{1/4}$.

Choose $\eps(\Lambda_{4}, \bC)$ sufficiently small so that if $|\lambda| \leq \Lambda_{4}$, then the nearest point projection from $B_\eps(\bC) \cap B_1 \setminus B_{1/4}$ onto $S_\lambda$ is smooth and lies in $S_\lambda \cap B_2 \setminus B_{1/8}$.  Ensure $\Lambda_{5}(\eps, \Lambda_{4}, \bC) \leq \Lambda_{4}$ is sufficiently small, so that $|\lambda| \leq \Lambda_{5}$ implies $S_\lambda \cap B_2 \setminus B_{1/8} \subset B_{\eps/2}(\bC)$.

Given $x \in B_1 \setminus B_{1/4}$, first assume that $x \not\in B_\eps(\bC)$.  In this case
\[
\eps/2 \leq d(x, S_\lambda) \leq 2,
\]
and hence we have
\[
\int_{(B_1\setminus B_{1/4}) \setminus B_\eps(\bC)} d_{S_{\lambda'}}^2 d\mu_M \leq 4 \mu_M(B_1) \leq (16/\eps^2) c(\bC) \int_{(B_1 \setminus B_{1/4}) \setminus B_\eps(\bC)} d_{S_\lambda}^2 d\mu_M.
\]

Now assume $x \in B_\eps(\bC)$.  Let $x'$ be the nearest point projection to $\bC$, and let $u(x') = x - x'$.  Since $|v_\lambda| + |\nabla v_\lambda| \leq c(\bC) |\Lambda_{5}|^{-\gamma}$ on $B_2 \setminus B_{1/8}$, we have
\[
d(x, S_{\lambda}) = (1+\psi(\Lambda_{5} | \bC)) |x - v_\lambda(x')| .
\]
In particular, ensuring $\Lambda_{5}(\bC)$ is sufficiently small and using Lemma \ref{lem:v-diff}, we get
\begin{align*}
d(x, S_{\lambda'}) 
&\leq 2 |x - v_{\lambda'}(x')| \\
&\leq 2|x - v_{\lambda}(x)| + 2 |v_{\lambda'}(x) - v_{\lambda}(x)| \\
&\leq 4 d(x, S_{\lambda}) + c(\bC) | (\lambda')^{1-\gamma} - \lambda^{1-\gamma}|.
\end{align*}
Integrating $d\mu_M$ over $B_1 \cap B_\eps(\bC) \setminus B_{1/4}$ gives the required result.
\end{proof}


\end{comment}

\section{$L^2$-Excess decay}

In this section we work towards the following decay theorem.
\begin{prop}[Decay Lemma]\label{prop:decay}
Given any $\theta \leq 1/8$, there are positive constants $\delta_{6}(\bC, \theta)$, $\Lambda_{6}(\bC, \theta)$, $c_{6}(\bC)$, $\alpha(\bC)$, so that the following holds: If $|\lambda| \leq \Lambda_{6}$, and $M$ is a stationary integral $n$-varifold in $B_1$, satisfying
\begin{equation}\label{eqn:decay-hyp1}
\int_{B_1} d_{S_\lambda}^2 d\mu_M \leq E \leq \delta_{6}^2, \quad \mu_M(B_1) \leq (7/4) \mu_\bC(B_1), \quad \mu_M(\overline{B_{1/10}}) \geq (1/2) \mu_\bC(B_{1/10}),
\end{equation}
then we can find $a \in \R^{n+1}$, $q \in SO(n+1)$, $\lambda' \in \R$, with
\begin{equation}\label{eqn:decay-concl1}
|a| + |q - Id| + |(\lambda')^{1-\gamma} - \lambda^{1-\gamma}| \leq c_{6} E^{1/2},
\end{equation}
so that
\begin{equation}\label{eqn:decay-concl2}
\theta^{-n-2} \int_{B_\theta(a)} d_{a+q(S_{\lambda'})}^2 d\mu_M \leq c_{6} \theta^{2\alpha} E, \quad \mu_M(\overline{B_{\theta/10}(a)}) \geq (1/2) \mu_\bC(B_{1/10}).
\end{equation}
\end{prop}

\vspace{5mm}

We first define a general notion of blow-up sequence and show how any blow-up sequence gives rise to a Jacobi field, i.e. a solution of the linearized problem $\cL_\bC(w) = 0$.

\begin{definition}
Consider the sequences $a_i \in \R^{n+1}$, $\lambda_i \in \R$, $q_i \in SO(n+1)$, $M_i$ stationary integral varifolds in $B_1 \subset \R^{n+1}$, and $E_i \in \R$.  We say the collection $(M_i, E_i, a_i, \lambda_i, q_i)$ is a blow-up sequence if:
\begin{enumerate}
\item $a_i \to 0$, $\lambda_i \to 0$, $q_i \to Id$, $E_i \to 0$

\item $\mu_{M_i}(B_1) \leq (7/4) \mu_\bC(B_1)$, $\mu_{M_i}(\overline{B_{1/10}}) \geq (1/2) \mu_\bC(B_{1/10})$

\item $\limsup_i E_i^{-1} \int_{B_1} d_{a_i + q_i(S_{\lambda_i})}^2 d\mu_{M_i} < \infty$
\end{enumerate}
\end{definition}

\begin{prop}\label{prop:blow-up}
Let $(M_i, \beta_i, a_i, \lambda_i, q_i)$ be a blow-up sequence.  From Theorem \ref{thm:non-conc}, there is a sequence of radii $\tau_i \to 0$, so that
\[
M \cap B_{1/2} \setminus B_{\tau_i} = \graph_{a_i + q_i(S_{\lambda_i})}(u_i), \quad \sum_{k=0}^3 r^{k-1} |\nabla^k u_i| \to 0,
\]
and
\[
(a_i + q_i(S_{\lambda_i})) \setminus B_{\tau_i} = \graph_\bC(\phi_i), \quad \sum_{k=0}^3 r^{k-1} |\nabla^k \phi_i| \to 0.
\]
Write $\Phi_i(x) = x + \phi_i(x) \nu_{\bC}$ for the graphing function associated to $\phi_i$.

There is a subsequence, also denoted $i$, and a solution $w :  \bC \cap B_{1/4} \to \R$ to $\cL_\bC(w) = 0$, satisfying the following:
\begin{enumerate}
\item smooth convergence $E_i^{-1/2} u_i \circ \Phi_i \to w$ on compact subsets of $\bC \cap B_{1/4} \setminus \{0\}$ ;

\item $L^\infty$ decay: for all $r < 1/4$:
\[
w(r\omega)^2 \leq c(\bC)r^{2\gamma}  \left( \limsup_i E_i^{-1} \int_{B_1} d_{a_i + q_i(S_{\lambda_i})}^2 d\mu_{M_i} \right) ;
\]
\item strong $L^2$ convergence:
\[
E_i^{-1} \int_{B_r} d_{a_i + q_i(S_{\lambda_i})}^2 d\mu_{M_i} \to \int_{B_r} w^2 d\mu_\bC \quad \forall r \leq 1/4.
\]
\end{enumerate}
\end{prop}

\begin{remark}
For shorthand, we will often say $E_i^{-1/2} u_i$ converges smoothly to $w$ to indicate convergence as in Proposition \ref{prop:blow-up}, conclusion 1.
\end{remark}

\begin{proof}
Part 1 is a fairly standard argument (see e.g. \cite{Simon1}), and parts 2, 3 follow directly from Theorem \ref{thm:non-conc}.  We outline the argument of part 1.  Since $a_i + q_i(S_{\lambda_i})$ converges smoothly to $\bC$ on compact subsets of $B_1 \setminus \{0\}$, the coefficients of $\cM_{a_i + q_i(S_{\lambda_i})}$, $\cL_{a_i + q_i(S_{\lambda_i})}$ converge locally smoothly to those of $\cM_\bC$, $\cL_\bC$.  Using this and standard elliptic estimates, we have for any compact $K \subset \subset (a_i + q_i(S_{\lambda_i})) \cap B_{1/2} \setminus \{0\}$, and $l = 0, 1, 2. \ldots$, uniform estimates of the form
\begin{align}
\sup_K |\nabla^l u_i| 
&\leq c(K, l) \left( \int_{(a_i + q_i(S_{\lambda_i})) \cap B_{1/2}} u_i^2 d\haus^n \right)^{1/2} \\
&\leq c(K, l) E_i^{1/2} \left( \limsup_i E_i^{-1} \int_{B_r} d_{a_i + q_i(S_{\lambda_i})}^2 d\mu_{M_i}\right)
\end{align}
Of course by assumption $\Phi_i \to 0$ in $C^\infty_{loc}(\bC \cap B_1 \setminus \{0\})$.  After passing to a subsequence, we deduce $C^\infty_{loc}(\bC \cap B_{1/4} \setminus \{0\})$ convergence of the functions $E_i^{-1/2} u_i \circ \Phi_i$ to some $w \in C^\infty(\bC \cap B_{1/4})$.

Now we can write
\[
0 = \cM_{a_i + q_i(S_{\lambda_i})}(u_i) = \cL_i(u_i) + \cE_i(u_i),
\]
where $\cL_i \equiv \cL_{a_i + q_i(S_{\lambda_i})}$ converges smoothly away from $0$ to the operator $\cL_\bC$, and where
\[
\sup_K |\cE_{S_{\lambda_i}}(u_i)| = C(K) |u|_{C^2(K)}^2 = o(1)E_i^{1/2}.
\]
It follows easily that $w$ solves the Jacobi operator $\cL_\bC(w) = 0$. 
\end{proof}

\vspace{5mm}

\begin{proof}[Proof of Proposition \ref{prop:decay}]
Choose $\alpha(\bC)$ so that $\gamma_3^+ = 1 < 1+\alpha \leq \gamma_{4}^+$.  Fix $\theta \leq 1/8$.  We first prove the decay estimate.  Suppose, towards a contradiction, there are sequences of numbers $\delta_i \to 0$, $\lambda_i \to 0$, $E_i \to 0$, and stationary integral varifolds $M_i$ in $B_1$, which satisfy:
\[
\int_{B_1} d_{S_{\lambda_i}}^2 d\mu_{M_i} \leq E_i \leq \delta_i, \quad \mu_{M_i}(B_1) \leq (7/4) \mu_\bC(B_1), \quad \mu_{M_i}(\overline{B_{1/10}}) \geq (1/2) \mu_\bC(B_{1/10}),
\]
but for which
\[
\theta^{-n-2} \int_{B_{\theta}(a)} d_{a' + q'(S_{\lambda'})}^2 d\mu_{M_i} \geq c_{6} \theta^{2\alpha} E_i
\]
for any $a \in \R^{n+1}$, $q \in SO(n+1)$, $\lambda' \in \R$ satisfying
\[
|a| + |q - Id| + |\lambda' - \lambda_i|^{1-\gamma} \leq c_{6} E_i^{1/2}.
\]
Here $c_{6}(\bC)$ will be fixed shortly.

For any sequence $\tau_i$, $\beta_i$ tending to $0$ sufficiently slowly, by Theorem \ref{thm:non-conc} we can write
\[
\spt M_i \cap B_{1/2} \setminus B_{\tau_i} = \graph_{S_{\lambda_i}}(u_i), \quad \sum_{k=0}^3 r^{k-1} |\nabla^k u_i| \leq \beta_i.
\]
By definition, $(M_i, E_i, 0, \lambda_i, Id)$ is a blow-up sequence, and so by Proposition \ref{prop:blow-up} there is a solution $w : \bC \cap B_{1/4} \to \R$ to $L w = 0$ satisfying:
\begin{equation}\label{eqn:w-bounds}
\int_{B_{1/4}} w^2 d\mu_\bC \leq 1, \quad |w(r\omega)| \leq c(\bC) r^\gamma,
\end{equation}
so that, after passing to a subsequence (also denoted $i$),
\[
E_i^{-1/2} u_i \to w
\]
smoothly on compact subsets of $\bC \cap B_{1/4} \setminus \{0\}$, and
\[
E_i^{-1} \int_{B_\rho} d_{S_{\lambda_i}}^2 d\mu_{M_i} \to \int_{B_\rho} w^2 d\mu_\bC \quad \forall \rho \leq 1/4.
\]

Using the pointwise bound \eqref{eqn:w-bounds} combined with \cite[(2.10) Lemma]{SiSo} to kill the modes $\gamma_i^-$, $i\in\N$, and the strongly integrable nature of $\bC$, there are $e \in \R$, $b \in \R^{n+1}$, and $A$ a skew-symmetric $(n+1)\times  (n+1)$ matrix, so that we can expand $w$ in $L^2(\bC \cap B_{1/4})$ as
\[
w(r\omega = x) = e r^{\gamma} + \nu_\bC(r\omega) \cdot (b + A x) + \sum_{j : \gamma_j^+ \geq 1+ \alpha} r^{\gamma_j^+} z_j(\omega)
\]
where the sum is $L^2(\Sigma)$-orthogonal for each fixed $r$.  In particular, using the $L^2$ bound \eqref{eqn:w-bounds} and an appropriate choice of $r \in (1/8, 1/4)$, we get
\[
|e| + |b| + |A| \leq c_{7}(\bC),
\]
and by Fubini we have
\begin{equation}\label{eqn:w-sum-bound}
\sum_{j : \gamma_j^+ \geq 1+\alpha} \frac{(1/4)^{2\gamma_j^++n}}{2\gamma_j^+ +n} \int_\Sigma z_j^2 \leq 1.
\end{equation}

\vspace{5mm}

We first show that, by replacing $\lambda_i$ with appropriate $\lambda_i'$, we can arrange so that $e = 0$.  Let us define
\[
\lambda_i' = (e E_i^{1/2} + \lambda_i^{1-\gamma})^{1/(1-\gamma)}.
\]
Trivially $\lambda_i' \to 0$.  Corollary \ref{cor:change-lambda} implies that $(M_i, E_i, 0, \lambda_i', Id)$ is a blow-up sequence also.

Our choice implies that
\[
v_{\lambda_i'} - v_{\lambda_i} = (1+o(1)) ((\lambda_i')^{1-\gamma} - \lambda_i^{1-\gamma}) r^\gamma = (1+o(1)) e E_i^{1/2} r^\gamma,
\]
where we write $o(1)$ to signify any function which tends to $0$ as $i \to \infty$.  We can write
\[
S_{\lambda_i'} \cap B_1 \setminus B_{\tau_i} = \graph_{S_{\lambda_i}}(v_i),
\]
where, setting $x = r\omega + v_{\lambda_i}(r\omega) \nu_\bC(r\omega)$:
\[
|v_i(x) - (v_{\lambda_i'} - v_{\lambda_i})(x)| \leq o(1) |(v_{\lambda_i'} - v_{\lambda_i})(x)|.
\]
In other words, 
\[
v_i(x) = (1+o(1)) eE_i^{1/2} r^\gamma.
\]

Write
\[
\spt M_i \cap B_{1/2} \setminus B_{\tau_i} = \graph_{S_{\lambda_i'}}(\tilde u_i).
\]
If we set $y = x + v_i(x) \nu_{S_{\lambda_i}}(x)$ (for $x$ as above), then
\[
\tilde u_i(y) = (1+o(1))(u_i(x) - v_i(x)) = (1+o(1))u_i(x) - (1+o(1)) eE_i^{1/2} r^\gamma.
\]
Applying Proposition \ref{prop:blow-up} to the blow-up sequence $(M_i, E_i, 0, \lambda_i', Id)$, we deduce that, after passing to a further subsequence, $E_i^{-1/2} \tilde u_i$ converges smoothly on compact subsets to
\[
w - e r^{\gamma} \equiv \nu_\bC(r\omega) \cdot (b + A x) + \sum_{j : \gamma_j^+ \geq 1+ \alpha} r^{\gamma_j^+} z_j(\omega),
\]
and we have strong $L^2$ convergence
\[
E_i^{-1} \int_{B_\rho} d_{S_{\lambda_i'}}^2 d\mu_{M_i} \to \int_{B_\rho} (w - e r^\gamma)^2 d\mu_\bC \quad \forall \rho \leq 1/4.
\]

\vspace{5mm}

We now show how to pick $a_i$, $q_i$ to arrange so that $b = 0$, $A = 0$.  This more standard, and essentially follows the usual ``integrability through rotations'' argument.

Choose $a_i = b E_i^{1/2}$ and $A_i = A E_i^{1/2}$, and let $q_i = \exp(A_i)$.  It's easy to check that $|a_i| + |q_i - Id| \leq c(\bC) E_i^{1/2}$.  Since
\[
d_H(S_{\lambda_i'} \cap B_2, (a_i + q_i(S_{\lambda_i'}) \cap B_2) \leq c(\bC) E_i^{1/2}
\]
it follows that $(M_i, E_i, a_i, \lambda_i', q_i)$ is a blow-up sequence also.

We can write
\[
(a_i + q_i(S_{\lambda_i'})) \cap B_1 \setminus B_{\tau_i} = \graph_{S_{\lambda_i'}}(v_i)
\]
(for $\tau_i \to 0$ sufficiently slowly) where
\[
v_i(x) = (1+o(1)) \nu_{S_{\lambda_i'}}(x) \cdot (a_i + A_i(x)).
\]

So now if $\tilde u_i$ is the graphing function of $M_i$ over $a_i + q_i(S_{\lambda_i'})$, and $u_i$ is the graphing function of $M_i$ over $S_{\lambda_i'}$, then we have
\[
|\tilde u_i( x + v_i(x)\nu_{S_{\lambda_i}}(x)) - (u_i(x) - v_i(x))| \leq o(1) |u_i(x) - v_i(x)|.
\]
This implies that
\begin{align*}
\tilde u_i( y ) &= (1+o(1))( u_i(x) - v_i(x)) \\
&= (1+o(1)) u_i(x) - (1+o(1)) \nu_{S_{\lambda_i'}}(x) \cdot (b + A(x)) E_i^{1/2}.
\end{align*}
Applying Proposition \ref{prop:blow-up} to this new blow-up sequence, we deduce that (after passing to a further subsequence) $E_i^{1/2} \tilde u_i$ converges smoothly on compact subsets to
\[
w - er^\gamma - \nu_\bC(r\omega) \cdot (b + A x) \equiv \sum_{j : \gamma_j^+ \geq 1+ \alpha} r^{\gamma_j^+} z_j(\omega),
\]
and we have strong $L^2$ convergence
\[
E_i^{-1} \int_{B_\rho} d_{a_i + q_i (S_{\lambda_i'})}^2 d\mu_{M_i} \to \int_{B_\rho} (w - e r^\gamma - \nu_\bC \cdot (b + Ax))^2 d\mu_\bC \quad \forall \rho \leq 1/4.
\]

\vspace{5mm}

We've demonstrated that by judiciously choosing our $a_i, q_i, \lambda_i'$, we can arrange so that
\[
w = \sum_{j : \gamma_j^+ \geq 1+ \alpha} r^{\gamma_j^+} z_j(\omega),
\]
where $z_j$ continue to satisfy the bound \eqref{eqn:w-sum-bound}.  Using \eqref{eqn:w-sum-bound}, and the fact that $4\theta \leq 1$, we compute:
\begin{align*}
\int_{B_{2\theta}} w^2 d\mu_\bC
&= \sum_{j : \gamma_j^+ \geq 1+\alpha} \frac{(2\theta)^{2\gamma_j^++n}}{2\gamma_j^++n} \int_\Sigma z_j^2(\omega) d\omega \\
&\leq \max_{j : \gamma_j^+ \geq 1+\alpha} (8\theta)^{2\gamma_j^++n} \\
&\leq (8\theta)^{n+2+2\alpha}.
\end{align*}
So by the strong $L^2$ convergence, for sufficiently large $i$ we must have
\[
\int_{B_{2\theta}} d_{a_i + q_i(S_{\lambda_i'})}^2 d\mu_{M_i} \leq 16^{n+2+2\alpha} \theta^{n+2+2\alpha} E_i,
\]
To recenter, we simply observe that for $i$ sufficiently large, we have
\[
B_{\theta}(a_i) \subset B_\theta,
\]
and hence we get
\[
\theta^{-n-2} \int_{B_{\theta}(a_i)} d_{a_i + q_i(S_{\lambda_i'})}^2 d\mu_{M_i} \leq 32^{n+4} \theta^{2\alpha} E_i,
\]
which is a contradiction for sufficiently large $i$.

\vspace{5mm}

Finally let us establish the lower volume bound \eqref{eqn:decay-concl2}.  This is also a straightforward proof by contradiction.  Suppose otherwise: there is a sequence $\delta_i \to 0$, $\lambda_i \to 0$, and $M_i$ satisfying the hypotheses \eqref{eqn:decay-hyp1} , and the decay of \eqref{eqn:decay-concl1}, \eqref{eqn:decay-concl2}, but for which
\[
\mu_M(\overline{B_{\theta/10}(a_i)}) < (1/2) \mu_\bC(B_{\theta/10})
\]
for all $i$.

By compactness of stationary varifolds, we can pass to a subsequence (also denoted $i$) so that $M_i \to M$ in $B_1$, for some integral stationary $n$-varifold $M$ in $B_1$.  Since
\[
\int_{B_1} d_{\bC}^2 d\mu_{M_i} \to 0,
\]
by the constancy theorem $M = k [\bC]$, for some integer $k$.  Since $\mu_{M}(\overline{B_{1/10}}) \geq (1/2) \mu_\bC(B_{1/10})$ for all $i$, we must have $k \geq 1$.  Since $\mu_{M}(B_1) \leq (7/4) \mu_\bC(B_1)$, we must have $k \leq 1$.  So in fact $k = 1$, and we deduce that $M_i$ varifold converge to $[\bC]$.

Since $|a_i| \to 0$, for any $0 < \eps < 1$ and sufficiently large $i$ we have
\[
\mu_{M_i}(\overline{B_{\theta/10}(a_i)}) \geq \mu_{M_i}(B_{(1-\eps)\theta/10}) \to \mu_\bC(B_{(1-\eps)\theta/10}) = (1-\eps)^n \mu_\bC(B_{\theta/10}).
\]
Choosing $\eps(n)$ sufficiently small, so that $(1-\eps)^n \geq 3/4$, and we obtain a contradiction for large $i$.  This completes the proof of Proposition \ref{prop:decay}
\end{proof}

\section{Regularity}

The key idea to obtain regularity is to iterate Proposition \ref{prop:decay} at decreasing scales, until $\lambda$ scale-invariantly becomes too big.  This is the radius at which we start to ``see'' the foliation as separate from the cone, and this is the radius at which we stop.  If no such radius exists, we keep iterating until radius $0$, to deduce regularity over the cone.  Note that, from only the information we start with, we have no way of predetermining how large this radius is.

\begin{prop}\label{prop:iterate}
There are $\beta(\bC)$, $\delta_{7}(\bC)$, $c_{7}(\bC)$ so that the following holds.  Take $|\lambda| \leq \Lambda_{6}$, and let $M$ be a stationary integral varifold in $B_1$ satisfying
\begin{equation}\label{eqn:iterate-hyp}
\int_{B_1} d_{S_\lambda}^2 d\mu_M \leq E \leq \delta_{7}^2, \quad \mu_M(B_1) \leq (3/2) \mu_\bC(B_1), \quad \mu_M(\overline{B_{1/10}}) \geq (1/2) \mu_\bC(B_{1/10}).
\end{equation}
Then there is a $a \in R^{n+1}$, $q \in SO(n+1)$, $\lambda' \in \R$, with
\begin{equation}\label{eqn:iterate-concl1}
|a| + |q - Id| + |(\lambda')^{1-\gamma} - \lambda^{1-\gamma}| \leq c_{7} E^{1/2},
\end{equation}
so that, for all $1 \geq r > c_{7} |\lambda'|$, we have the decay:
\begin{equation}\label{eqn:iterate-concl2}
r^{-n-2} \int_{B_r(a)} d_{a + q(S_{\lambda'})}^2 d\mu_M \leq c_{7} r^{2\beta} E,
\end{equation}
and the volume bounds
\begin{equation}\label{eqn:iterate-concl3}
\mu_M(B_r(a)) \leq (7/4) \mu_\bC(B_r), \quad \mu_M(\overline{B_{r/10}(a)}) \geq (1/c_{7}) \mu_\bC(B_{r/10}).
\end{equation}
\end{prop}

\begin{remark}
In fact one can take $\beta$ to be anything in the interval $(0, \alpha)$, except of course in this case the various constants $\delta_{7}$, $c_{7}$ will depend on the choice of $\beta$ also.
\end{remark}

\begin{proof}
Choose $\theta(\bC) \leq 1/4$ sufficiently small so that $c_{6}\theta^{2\alpha} \leq 1/4$.  Set $r_i = \theta^i$.  We claim that we can find an integer $I \leq \infty$, and sequences $a_i \in \R^{n+1}$, $q_i \in SO(n+1)$, $\lambda_i \in \R$, ($i = 0, 1, \ldots, I$), so that for all $i < I$ we have:
\begin{align}
&a_0 = 0, \quad q_0 = Id, \quad \lambda_0 = \lambda, \label{eqn:a_0} \\
&r_i^{-1} |a_{i+1} - a_i| + |q_{i+1} - q_i| + r_i^{-1} |(\lambda_{i+1})^{1-\gamma} - (\lambda_i)^{1-\gamma}| \leq c_{6} 2^{-i} E^{1/2}, \label{eqn:a_i-bounds}\\
&|\lambda_i| \leq \Lambda_{6} r_i, \label{eqn:l_i-bounds}
\end{align}
and the decay:
\begin{equation}\label{eqn:r_i-decay}
r_i^{-n-2} \int_{B_{r_i}(a_i)} d_{a_i + q_i(S_{\lambda_i})}^2 d\mu_M \leq 4^{-i} E,
\end{equation}
and the volume bounds:
\begin{equation}\label{eqn:r_i-volume}
\mu_M(B_{r_i}(a_i)) \leq (7/4) \mu_\bC(B_{r_i}), \quad \mu_M(\overline{B_{r_{i}/10}(a_i)}) \geq (1/2) \mu_\bC(B_{r_i}).
\end{equation}
Moreover, if $I < \infty$, then 
\[
|\lambda_I| > \Lambda_{6} r_I.
\]

Let us prove this by induction.  Let us first show how the upper volume bound of \eqref{eqn:r_i-volume} follows from \eqref{eqn:a_i-bounds}.  If we have $a_0, \ldots, a_i$, satisfying \eqref{eqn:a_i-bounds}, then 
\[
|a_i| \leq \sum_{j=0}^{i-1} |a_{j+1} - a_j| \leq c_{6} E^{1/2} \sum_{j=0}^{i-1} r_i \leq 2 c_{6} \delta_{7}.
\]
Therefore, provided $\delta_{7}(\bC)$ is sufficiently small, we have by volume monotonicity
\begin{align*}
\mu_M(B_{r_i}(a_i)) 
&\leq (1-2c_{6} \delta_{7})^{-n} \mu_M(B_1) r_i^n \\
&\leq (1-2c_{6} \delta_{7})^{-n}(3/2) \mu_\bC(B_1) r_i^n \\
&\leq (7/4) \mu_\bC(B_{r_i}).
\end{align*}
Let us also ensure $\delta_{7}(\bC)$ is sufficiently small so that $2c_{6} \delta_{7} < \theta$.

It remains to show the existence of the $a_i, q_i, \lambda_i$.  Since $|\lambda_0| = |\lambda| \leq \Lambda_{6}$, we can apply Proposition \ref{prop:decay} to $M$ to obtain  $a_1, q_1, \lambda_1$, which satisfy the required estimates.  If $|\lambda_1| > \Lambda_{6} r_1$, we set $I = 1$ and stop.  This proves the base case of our induction.

Suppose, by inductive hypothesis, we have found $a_i, q_i, \lambda_i$ satisfying \eqref{eqn:a_i-bounds}, \eqref{eqn:l_i-bounds}, \eqref{eqn:r_i-decay}, \eqref{eqn:r_i-volume}, and for which $|\lambda_i| \leq \Lambda_6 r_i$.  By inductive hypotheses, we can apply Proposition \ref{prop:decay} to the varifold $M_i = (q_i^{-1})_\sharp (\eta_{a_i, r_i})_\sharp M$ and foliate $S_{r_i^{-1} \lambda_i}$ to obtain $\tilde a_{i+1}, \tilde q_{i+1}, \tilde \lambda_{i+1}$ satisfying \eqref{eqn:decay-concl1}, \eqref{eqn:decay-concl2}.  If we let
\[
a_{i+1} =  q_i(r_i \tilde a_{i+1}) + a_i , \quad q_{i+1} = q_i \circ \tilde q_{i+1}, \quad \lambda_{i+1} = r_i \tilde\lambda_{i+1}, 
\]
then it follows by scaling that this $a_{i+1}, q_{i+1}, \lambda_{i+1}$ satisfies the requirement estimates.  If $|\lambda_{i+1}| > \Lambda_{6} r_{i+1}$, we stop and set $I = i+1$.  Otherwise, continue.  By mathematical induction this proves the existence of the sequence.

\vspace{5mm}

If $I < \infty$, then let $a = a_I$, $q = q_I$, and $\lambda' = \lambda_I$.  Otherwise, observe that \eqref{eqn:a_i-bounds} imply that $a_i$, $q_i$ form a Cauchy sequence, and hence we take $a = \lim_i a_i$, $q = \lim_i q_i$, $\lambda' = 0$.

From \eqref{eqn:a_i-bounds}, we have for every $i < I$:
\[
r_i^{-1} |a - a_i| + |q - q_i| + r_i^{-1} | (\lambda')^{1-\gamma} - (\lambda_i)^{1-\gamma}| \leq 2 c_{6} 2^{-i} E^{1/2}.
\]
In particular, taking $i = 0$ gives \eqref{eqn:iterate-concl1}.

Given any $r_I \leq r < 1$, choose integer $i \leq I$ so that $r_{i+1} \leq r < r_i$.  Then, ensuring that $c_{6} \delta_{7} << \theta$, and using Corollary \ref{cor:change-lambda}, we have
\begin{align}
&r^{-n-2} \int_{B_r(a)} d_{a + q(S_{\lambda'})}^2 d\mu_M \nonumber\\
&\leq c r_i^{-2} |a_{i+1} - a_i|^2 + c |q_{i+1} - q_i|^2 + c r_i^{-2} ( (\lambda')^{1-\gamma} - (\lambda_i)^{1-\gamma})^2 \nonumber\\
&\quad + c r_i^{-n-2} \int_{B_{r_i}(a_i)} d_{a_i + q_i(S_{\lambda_i})}^2 d\mu_M \nonumber\\
&\leq c(\bC, \theta) 4^{-i} E \nonumber\\
&\leq c(\bC, \theta) r^{2\beta} E. \label{eqn:decay-for-r}
\end{align}
where $\beta = \log(1/2) / \log(\theta) > 0$.  Finally, observe that \eqref{eqn:a_i-bounds}, \eqref{eqn:l_i-bounds} imply
\[
\Lambda_{6} r_I < |\lambda'| \leq \Lambda_{6} r_{I-1} + (c_{6} \delta_{7})^{1/(1-\gamma)} r_{I-1} \leq c(\bC) r_I.
\]
Therefore, up to changing enlarging our constant $c$, we can take $r \geq \lambda'$ in \eqref{eqn:decay-for-r}.  The volume bounds follow directly from \eqref{eqn:r_i-volume}, \eqref{eqn:a_i-bounds}, and monotonicity.  This finishes the proof of Proposition \ref{prop:iterate}. 
\end{proof}

The proof of Theorem \ref{thm:main} is now essentially a straightforward application of Allard's theorem.
\begin{proof}[Proof of Theorem \ref{thm:main}]
Ensure $\delta_{1} \leq \delta_{7}$, and take $\Lambda_{1} = \Lambda_{6}$.  Apply Proposition \ref{prop:iterate} to obtain $a, q, \lambda'$.  For every $c_{7} |\lambda'| \leq r \leq 1$, a straightforward contradiction argument as in the proof of Theorem \ref{thm:non-conc}, implies that
\[
\spt M \cap B_{r/2}(a) \setminus B_{r/100}(a) = \graph_{a + q(S_{\lambda'})}(u), \quad r^{-1} |u| + |\nabla u| + [\nabla u]_{\beta, r} \leq c(\bC) r^{\beta} E^{1/2},
\]
for some $u : (a + q(S_{\lambda'})) \cap B_{r/2}(a) \setminus B_{r/200}(a) \to \R$.  If $\lambda' = 0$ we are done.

Suppose $\lambda' \neq 0$.  After scaling up by $|\lambda'|$, it suffices to show that there is a $u : (a+ q_{\sign(\lambda')}) \cap B_{c_{7}/2}(a) \to \R$ so that
\[
\spt M \cap B_{c_{7}/2}(a) = \graph_{a + q(S_{\sign(\lambda')})}(u), \quad |u| + |\nabla u| + [\nabla u]_\beta \leq c(\bC) E^{1/2},
\]
provided
\[
c_{7}^{-n-2} \int_{B_{c_{7}}(a)} d_{a + q(S_{\sign(\lambda')})}^2 d\mu_M \leq c_{7} E
\]
and
\[
\mu_M(B_{c_{7}}(a)) \leq (7/4) \mu_\bC(B_{c_{7}}), \quad \mu_M(\overline{B_{c_{7}/10}(a)}) \geq (1/c_{7}) \mu_\bC(B_{c_{7}/10}).
\]
However this is also an easy contradiction argument, taking $E \leq \delta_{1}$ to zero.
\end{proof}


\section{Mean curvature}

When $M$ has mean curvature, then we cannot use the maximum principle to conclude as in Proposition \ref{prop:decay}.  However, provided the mean curvature is sufficiently small and $L^2$ excess sufficiently large, then we can still get decay to one scale (Proposition \ref{prop:mc-decay}), by a straightforward contradiction argument, which suffices to characterize the singular nature, if not the precise local structure.

Iterating this gives scale-invariant smallness of the excess (Proposition \ref{prop:mc-iterate}), rather than decay.  At each scale we can deduce closeness to some rotate/translate of the cone $\bC$ (and hence foliate $S_\lambda$), but we cannot deduce that the rotations form a Cauchy sequence as we progress in scale.  If at some scale $\lambda$ becomes big, we stop, and we can deduce that $M \cap B_{1/2}$ is a $C^{1,\alpha}$ deformation of $S_1$ (\emph{without} effective estimates).  If we continue all the way to scale $0$, then we deduce that $M \cap B_{1/2}$ has an isolated singularity modeled on $\bC$, and a posteriori by \cite{AllAlm} we can deduce $M \cap B_{1/2}$ is a $C^{1,\alpha}$ graph over $\bC$.

\begin{prop}\label{prop:mc-decay}
For any $\theta \in (0, 1/4)$, $E_0 > 0$, there is a $\eps_{8}(\bC, \theta, E_0)$ so that the following holds.  Assume that $M$ is an integral $n$-varifold in $B_1$, satisfying $||\delta M||(B_1) \leq \eps_{8}$, and for some $E_0 \leq E$, $|\lambda| \leq \Lambda_{6}$ we have
\begin{equation}\label{eqn:mc-decay-hyp}
\int_{B_1} d_{S_\lambda}^2 d\mu_M \leq E \leq \delta_{6}^2, \quad \mu_M(B_1) \leq (7/4)\mu_\bC(B_1), \quad \mu_M(\overline{B_{1/10}}) \geq (1/2) \mu_\bC(B_{1/10}).
\end{equation}

Then there is an $a \in \R^{n+1}$, $q \in SO(n+1)$, $\lambda' \in \R$, satisfying
\begin{equation}\label{eqn:mc-decay-concl1}
|a| + |q - Id| + |(\lambda')^{1-\gamma} - \lambda^{1-\gamma}| \leq c_{6} E^{1/2},
\end{equation}
so that
\begin{equation}\label{eqn:mc-decay-concl2}
\theta^{-n-2} \int_{B_\theta(a)} d_{a + q(S_{\lambda'})}^2 d\mu_M \leq c_{6} \theta^{2\alpha} E, \quad \mu_M(\overline{B_{\theta/10}(a)}) \geq (1/2) \mu_\bC(B_{\theta/10}).
\end{equation}
Here $\delta_{6}(\bC, \theta)$, $c_{6}(\bC)$ are the constants from Theorem \ref{prop:decay}.
\end{prop}

\begin{proof}
Suppose, towards a contradiction, there is a sequence of integral $n$-varifolds $M_i$, and numbers $\eps_i \to 0$, $E_i \in [E_0, \delta_{6}^2]$, so that $M_i$ satisfy \eqref{eqn:mc-decay-hyp} and $||\delta M_i||(B_1) \leq \eps_i$, but for which \eqref{eqn:mc-decay-concl2} fails for all $a, q, \lambda'$ satisfying \eqref{eqn:mc-decay-concl1}.

We can pass to a subsequence, also denoted $i$, and obtain varifold converge $M_i \to M$, and convergence $E_i \to E \in [E_0, \delta_{6}^2]$.  The resulting $n$-varifold $M$ is stationary in $B_1$, continues to satisfy \eqref{eqn:mc-decay-hyp}, but fails \eqref{eqn:mc-decay-concl2} for all $a, q, \lambda'$ satisfying \eqref{eqn:mc-decay-concl1}.  However, $M$ satisfies the hypotheses of Proposition \ref{prop:decay}, contradicting the conclusions of Proposition \ref{prop:decay}.  This proves Proposition \ref{prop:mc-decay}.
\end{proof}

\begin{prop}\label{prop:mc-iterate}
There are constants $\delta_{9}(\bC)$, $c_{9}(\bC)$ which yield the following.  Given $E \in (0, \delta_{9}^2]$ and $p > n$, we can find an $\eps_{9}(\bC, E, p)$ so that: If $|\lambda| \leq \Lambda_{6}$, and $M$ is an integral $n$-varifold in $B_1$ with generalized mean-curvature $H_M$, with zero generalized boundary, and which satisfies
\[
\int_{B_1} d_{S_\lambda}^2 d\mu_M \leq E, \quad \mu_M(B_1) \leq (3/2) \mu_\bC(B_1), \quad \mu_M(\overline{B_{1/10}}) \geq (1/2) \mu_\bC(B_{1/10}),
\]
and
\[
\left( \int_{B_1} |H_M|^p d\mu_M \right)^{1/p} \leq \eps_{9}.
\]
Then there is an $a \in \R^{n+1}$, $\lambda' \in \R$, and for each $1 \geq r \geq c_{9} |\lambda'|$ there is a $q_r \in SO(n+1)$, which satisfy
\[
|a| + |\log(r)|^{-1} |q_r - Id| + |(\lambda')^{1-\gamma} - \lambda^{1-\gamma}| \leq c_{9} E^{1/2},
\]
so that for all $1 \geq r \geq c_{9} |\lambda'|$ we have the smallness:
\[
r^{-n-2} \int_{B_r(a)} d_{a+q_r(S_{\lambda'})}^2 d\mu_M \leq c_{9} E,
\]
and volume bounds
\[
\mu_M(B_r(a)) \leq (7/4) \mu_\bC(B_r), \quad \mu_M(\overline{B_{r/10}(a)}) \geq (1/c_{9})\mu_\bC(B_{r/10}).
\]
\end{prop}

\begin{proof}
The proof is almost verbatim to Proposition \ref{prop:iterate}, except we use Proposition \ref{prop:mc-decay} in place of Proposition \ref{prop:decay}.  Notice that Proposition \ref{prop:mc-decay} requires a lower bound on $E$, and so we cannot deduce decay of the $L^2$ excess, only smallness.   Choose $\theta(\bC)$ as in Proposition \ref{prop:mc-decay}, and set $r_i = \theta^i$.  We claim that we can find an integer $I \leq \infty$, and sequences $a_i \in \R^{n+1}$, $q_i \in SO(n+1)$, $\lambda_i \in \R$ ($i = 0, \ldots, I$) so that for all $i < I$ we have:
\begin{align}
&a_0 = 0, \quad q_0 = Id, \quad \lambda_0 = 0, \label{eqn:mc-a_0} \\
&r_i^{-1} |a_{i+1} - a_i| + |q_{i+1} - q_i| + r_i^{-1} |(\lambda_{i+1})^{1-\gamma} - (\lambda_i)^{1-\gamma}| \leq c_{6} E^{1/2}, \label{eqn:mc-a_i-bounds}\\
&|\lambda_i| \leq \Lambda_{6} r_i, \label{eqn:mc-l_i-bounds}
\end{align}
and the smallness:
\begin{equation}\label{eqn:mc-r_i-decay}
r_i^{-n-2} \int_{B_{r_i}(a_i)} d_{a_i + q_i(S_{\lambda_i})}^2 d\mu_M \leq E, \quad r_i^{1-n/p} \left( \int_{B_{r_i}(a_i)} |H_M|^p d\mu_M \right)^{1/p} \leq \eps_{9}, 
\end{equation}
and the volume bounds
\begin{equation}\label{eqn:mc-r_i-volume}
\mu_M(B_{r_i}(a_i)) \leq (7/4) \mu_\bC(B_{r_i}), \quad \mu_M(\overline{B_{r_i/10}(a_i)}) \geq (1/2) \mu_\bC(B_{r_i}).
\end{equation}

\vspace{5mm}

We proceed by induction.  Ensure $\delta_{9} \leq \delta_{7}(\bC)$.  As in the proof of Proposition \ref{prop:iterate}, given $a_0, \ldots, a_i$ satisfying \eqref{eqn:mc-a_i-bounds}, we have $|a_i| \leq 2c_{6} \delta_{9}$, and hence by volume monotonicity \eqref{eqn:monotonicity}
\begin{align*}
\mu_M(B_{r_i}(a_i)) &\leq ( 1  + c(p, \bC) \eps_{9} ) (1-2c_{6} \delta_{9})^{-n} \mu_M(B_1) r_i^{n} \leq (7/4) \mu_\bC(B_{r_i})
\end{align*}
provided $\delta_{9}(\bC)$ and $\eps_{9}(\bC, p)$ are sufficiently small.  This proves the upper volume bounds \eqref{eqn:mc-r_i-volume}.  The mean curvature bound in \eqref{eqn:mc-r_i-decay} follows trivially from $p > n$.

To obtain the $a_i, q_i, \lambda_i$, we first observe that, given any $a_i, q_i$ above, if we set $M_i = (q_i^{-1})_\sharp (\eta_{a_i, r_i})_\sharp M$, then
\begin{align*}
||\delta M_i||(B_1) 
&= r_i^{-n+1} \int_{B_{r_i}(a_i)} |H_M| d\mu_M \\
&\leq r_i^{-n+1} \left( \int_{B_{r_i}(a_i)} |H_M|^p d\mu_M \right)^{1/p} \mu_M(B_{r_i}(a_i))^{1-1/p} \\
&\leq c(\bC) \eps_{9}.
\end{align*}
In particular, ensuring $\eps_{9}(\bC, p)$ is sufficiently small, we can ensure that $||\delta M_i||(B_1) \leq \eps_{8}(\bC, p)$.  We can therefore proceed like the proof of Proposition \ref{prop:iterate}, using Proposition \ref{prop:mc-decay} in place of Proposition \ref{prop:decay}.
\end{proof}

\begin{proof}[Proof of Corollary \ref{cor:mc}]
Ensuring $\delta_{3} < \delta_{9}$, $\eps_{3} < \eps_{9}(\delta, p)$, then we can obtain $a, \lambda'$ and $q_r$ as in Proposition \ref{prop:mc-decay}.  A straightforward argument by contradiction, like in the proof of Theorem \ref{thm:main}, shows that for every $1 \geq r \geq c_{9} |\lambda'|$,
\[
\spt M \cap B_{r/2}(a) \setminus B_{r/100}(a) = \graph_{a + q_r(S_{\lambda'})}(u), \quad r^{-1} |u| + |\nabla u| + [\nabla u]_{\beta, r} \leq c(\bC) \delta ,
\]
(provided $\delta_{3}, \eps_{3}$ are sufficiently small, depending on $\bC$, $p$).  If $\lambda' = 0$, this shows that any tangent cone at $a$ is rotation of $\bC$.  Therefore, shrinking $\delta_{3}, \eps_{3}$ as necessary, \cite{AllAlm} implies $\spt M \cap B_{1/2}$ is a $C^{1, \beta}$ perturbation of $\bC$.

If $\lambda' \neq 0$, then arguing as in the proof of Theorem \ref{thm:main}, we get that
\[
\spt M \cap B_{c_{9} |\lambda'|/2}(a) = \graph_{a + q_{c_{9}|\lambda'|}(S_{\lambda'})}(u), \quad r^{-1} |u| + |\nabla u| + [\nabla u]_\beta \leq c(\bC) \delta.
\]
In particular, $\spt M \cap B_{1/2}$ is regular.
\end{proof}

\end{comment}

\section{Corollaries, related results}

In this section we give an alternate proof of \cite[Theorem 0.3]{SiSo} (i.e., Corollary \ref{cor:ss} of this paper) using our main regularity Theorem \ref{thm:main}.  We also prove, using results of \cite{CaHaSi}, a uniqueness result for minimal graphs over $\bC$ (Proposition \ref{prop:uniqueness}).

\begin{proof}[Proof of Corollary \ref{cor:ss}]
Let
\[
E_i = R_i^{-n-2} \int_{B_{R_i}} d_{\bC}^2 d\mu_M
\]
By hypothesis, $E_i \to 0$.  For all $i$ sufficiently large, we can apply Theorem \ref{thm:main} to deduce the existence of $a_i \in \R^{n+1}, q_i \in SO(n+1), \lambda_i \in \R$, and $u_i : (a_i+q_i(S_{\lambda_i})) \cap B_{R_i}(a_i) \to \R$, so that for all $c_{7}|\lambda_i| \leq r \leq R_i$, 
\begin{equation}\label{eqn:ss-graph}
\spt M \cap B_{r}(a_i) = \graph_{a_i + q_i(S_{\lambda_i})}(u_i), \quad r^{-1}|u_i| + |\nabla u_i| \leq c(\bC) E_i^{1/2}.
\end{equation}

Suppose, towards a contradiction, that $a_i \to \infty$.  Let $U_i(x) = x + u_i(x) \nu_{a_i + q_i(S_{\lambda_i})}(x)$ be the graphing function associated to $u_i$.  From \eqref{eqn:ss-graph} we have that $|U_i(x) - x| \leq o(1)|x - a_i|$.

Fix any $\rho > 0$, then by the previous paragraph we have
\[
U_i^{-1}(\spt M \cap B_\rho) \subset (a_i + q_i(S_{\lambda_i})) \cap B_{2|a_i|}(a_i) \setminus B_{|a_i|/2}(a_i)
\]
for $i$ sufficiently large.  Now the curvature of $(a_i + q_i(S_{\lambda_i})) \cap B_{2|a_i|}(a_i) \setminus B_{|a_i|/2}(a_i)$ tends to zero uniformly as $i \to \infty$, and $|\nabla u_i| = o(1)$, and so we must have that $\spt M \cap B_\rho$ is contained in a plane.  Taking $\rho \to \infty$, we deduce $\spt M$ is planar, and hence $\bC$ is planar also.  This is a contradiction, and so $a_i$ must be bounded.

By \eqref{eqn:ss-graph}, we have
\[
d(0, \spt M) = d(0, a_i + q_i(S_{\lambda_i})) + o(1) |a_i| \geq \frac{1}{c(\bC)} |\lambda_i| - c(\bC) |a_i|
\]
for $i$ large.  Since $a_i$ is bounded, we must have that $\lambda_i$ is bounded also.  We can pass to a subsequence, also denoted $i$, so that $a_i \to 0$, $\lambda_i \to \lambda$, and $q_i \to q \in SO(n+1)$.  We get smooth convergence on compact subsets of $\R^{n+1} \setminus \{a\}$
\[
a_i + q_i(S_{\lambda_i}) \to a + q(S_\lambda),
\]
and $C^1_{loc}$ convergence $u_i \to 0$.  We deduce that $\spt M = a + q(S_\lambda)$, which is the desired conclusion.
\end{proof}

\subsection{Uniqueness of graphs over $\bC$}\label{sec:uniqueness}

\cite{CaHaSi} prove the following theorem constructing a plethora of examples of minimal surfaces in $B_1$, which are perturbations of a given minimal cone.  To state their theorem properly we need a little notation.  Given $J \in \N$, define the projection mapping $\Pi_J : L^2(\Sigma) \to L^2(\Sigma)$
\[
\Pi_J(g)(r\omega) = \sum_{j \geq J+1} <g, \phi_j>_{L^2(\Sigma)} r^{\gamma_j^+} \phi_j(\omega) .
\]
For shorthand write $\bC_1 = \bC \cap \overline{B_1}$.

\begin{theorem}[\cite{CaHaSi}]
Take $m > 1$, $\alpha \in (0,1)$, and $J \in \N$ so that $\gamma_J^+ \leq m < \gamma_{J+1}^+$.  There are $\eps(\bC, m, \alpha)$, $\Lambda(\bC, m, \alpha)$ so that given any $g \in C^{2,\alpha}$ satisfying $|g|_{2,\alpha} \leq \eps$, and any $\lambda \in (0, \Lambda)$, then there is a solution $u_\lambda \in C^{2,\alpha}(\bC_1)$ to the problem
\[
M_\bC(u) = 0 \text{ on } \bC_1, \quad \Pi_J(u_\lambda) = \lambda \Pi_J(g) \text{ on } \Sigma,
\]
satisfying
\[
r^{-m} |u|_{2,\alpha,r} \leq c(\bC, m, \alpha) \lambda |g|_{2,\alpha} \quad \forall 0 < r \leq 1.
\]
\end{theorem}

Loosely speaking, \cite{CaHaSi} are solving a boundary-value-type problem, where one is allowed to specify the decay rate at $r = 0$, and the Fourier modes at $r = 1$ which decay \emph{faster} than the prescribed rate.  Though they do not comment on it, implicit in their work is the uniqueness statement of Proposition \ref{prop:uniqueness}.  The basic idea is that solutions to the minimal surface operator can be written as fixed points to a contraction mapping, provided the solutions decay like $r^{1+\eps}$, and have boundary data sufficiently small.  We illustrate this below.

\begin{proof}[Proof of Proposition \ref{prop:uniqueness}]
We use the notation of Theorem \ref{thm:main}.  Fix an $\alpha \in (0,1)$, and given $w \in C^{2,\alpha}(\bC_1)$, define the norm
\begin{equation}\label{eqn:u-B-bound}
|w|_B = \sup_{0 < r \leq 1} r^{-\beta} |u|_{2,\alpha,r}.
\end{equation}
Pick $J$ so that $\gamma_J^+ = 1$, and recall that $1+\beta \in (\gamma_J^+, \gamma_{J+1}^+)$.

By assumption, we have $\spt M \cap B_{1/2} = \graph_{U \subset a + q(\bC)}(u)$, where $u$ satisfies the estimates \eqref{eqn:main-concl2}.  Since we are concerned with $M \cap B_{1/4}$, ensuring $\delta_{2}(\bC)$ is small, there is no loss in assuming (after scaling, translating, rotating) that $\spt M = \graph_{\bC_1}(u)$.  By standard interior elliptic estimates and decay \eqref{eqn:main-concl2}, we can assume that $u$ is smooth, and satisfies
\[
|u|_B \leq c(\bC,\alpha) \delta_{2}.
\]

We can write the mean curvature operator as 
\[
\cM_\bC(u) = \cL_\bC(u) + \cE(u),
\]
where the non-linear error part $\cE(u)$ satisfies certain, relatively standard scale-invariant structure conditions (see \cite{CaHaSi}).  Given $g \in C^{2,\alpha}(\Sigma)$, \cite{CaHaSi} show there are numbers $\eps_{10} < \delta_{10}$, depending only on $\bC, \alpha$, so that provided $|g|_{2,\alpha} < \eps_{10}$, then for every $w$ in the convex space
\[
\cB := \{ w \in C^{2,\alpha}(\bC_1) : \quad \Pi_J(w_{r =1}) = \Pi_J(g), \quad |w|_B \leq \delta_{10}\},
\]
there is a unique solution $v = :\cU(w) \in \cB$ to the linear problem
\[
\cL_\bC(v) = -\cE(w) \text{ on } \bC_1, \quad \Pi_J(v_{r = 1}) = \Pi_J(g) \text{ on } \Sigma, \quad \sup_{0 < r \leq 1} r^{-\beta} \left( \int_{\Sigma} v(r\omega)^2 d\omega \right)^{1/2} < \infty ,
\]
and \emph{moreover}, that $\cU$ is a contraction mapping on $\cB$. 

To prove Proposition \ref{prop:uniqueness} it therefore suffices to show that, with $g = u|_{r =1}$, then $u \in \cB$ and $|g|_{2,\alpha} \leq \eps_{10}$.  However, both of these follow from \eqref{eqn:u-B-bound} by ensuring $\delta_{2}(\bC, \alpha)$ is sufficiently small.
\end{proof}

\bibliographystyle{plain}
\bibliography{references-Cal}

\begin{thebibliography}{10}

\bibitem{All}
William~K. Allard.
\newblock On the first variation of a varifold.
\newblock {\em Ann. of Math. (2)}, 95:417--491, 1972.

\bibitem{AllAlm}
William~K. Allard and Frederick~J. Almgren, Jr.
\newblock On the radial behavior of minimal surfaces and the uniqueness of
  their tangent cones.
\newblock {\em Ann. of Math. (2)}, 113(2):215--265, 1981.

\bibitem{becker-kahn2017}
Spencer~T. Becker-Kahn.
\newblock Transverse singularities of minimal two-valued graphs in arbitrary
  codimension.
\newblock {\em J. Differential Geom.}, 107(2):241--325, 10 2017.

\bibitem{BDG}
E.~Bombieri, E.~De~Giorgi, and E.~Giusti.
\newblock Minimal cones and the {B}ernstein problem.
\newblock {\em Invent. Math.}, 7:243--268, 1969.

\bibitem{CaHaSi}
Luis Caffarelli, Robert Hardt, and Leon Simon.
\newblock Minimal surfaces with isolated singularities.
\newblock {\em Manuscripta Math.}, 48(1-3):1--18, 1984.

\bibitem{CoEdSp}
Maria {Colombo}, Nick {Edelen}, and Luca {Spolaor}.
\newblock {The singular set of minimal surfaces near polyhedral cones}.
\newblock {\em arXiv e-prints}, page arXiv:1709.09957, Sep 2017.

\bibitem{HaSi}
Robert Hardt and Leon Simon.
\newblock Area minimizing hypersurfaces with isolated singularities.
\newblock {\em J. Reine Angew. Math.}, 362:102--129, 1985.

\bibitem{Simoes}
Plinio Simoes.
\newblock On a class of minimal cones in {${\bf R}^{n}$}.
\newblock {\em Bull. Amer. Math. Soc.}, 80:488--489, 1974.

\bibitem{Simon0}
Leon Simon.
\newblock Asymptotics for a class of nonlinear evolution equations, with
  applications to geometric problems.
\newblock {\em Ann. of Math. (2)}, 118(3):525--571, 1983.

\bibitem{Simon1}
Leon Simon.
\newblock Cylindrical tangent cones and the singular set of minimal
  submanifolds.
\newblock {\em J. Differential Geom.}, 38(3):585--652, 1993.

\bibitem{Simon-CAG}
Leon Simon.
\newblock Uniqueness of some cylindrical tangent cones.
\newblock {\em Communications in Analysis and Geometry}, 2:1--33, 1994.

\bibitem{SiSo}
Leon Simon and Bruce Solomon.
\newblock Minimal hypersurfaces asymptotic to quadratic cones in {${\bf
  R}^{n+1}$}.
\newblock {\em Invent. Math.}, 86(3):535--551, 1986.

\bibitem{SoWh}
Bruce Solomon and Brian White.
\newblock A strong maximum principle for varifolds that are stationary with
  respect to even parametric elliptic functionals.
\newblock {\em Indiana Univ. Math. J.}, 38(3):683--691, 1989.

\end{thebibliography}

\end{document}